\documentclass[12pt]{amsart}%
\usepackage{amsmath,amsfonts,amsthm,amscd,amssymb,stmaryrd,mathrsfs}%
\usepackage{amsmath}
\usepackage{amsfonts}
\usepackage{amssymb}
\usepackage{verbatim}
\usepackage{graphicx}
\usepackage{color}%
\providecommand{\U}[1]{\protect\rule{.1in}{.1in}}
\newtheorem{theorem}{Theorem}

\newtheorem{claim}[theorem]{Claim}

\newtheorem{corollary}[theorem]{Corollary}

\newtheorem{definition}[theorem]{Definition}

\newtheorem{lemma}[theorem]{Lemma}
\newtheorem{notation}[theorem]{Notation}

\newtheorem{proposition}[theorem]{Proposition}
\newtheorem{remark}[theorem]{Remark}

\newcommand{\R}{\mathbb{R}}

\numberwithin{equation}{section}
\numberwithin{theorem}{section}
\begin{document}
\title[Approximating coarse Ricci]{Approximating coarse Ricci curvature on submanifolds of Euclidean space}
\author{Antonio G. Ache}
\email{tonyache@gmail.com}
\author{Micah W. Warren}
\address{Department of Mathematics, University of Oregon, Eugene OR 97403}
\email{micahw@oregon.edu}
\thanks{The first author was partially supported by a postdoctoral fellowship of the
National Science Foundation, award No. DMS-1204742.}
\thanks{The second author is partially supported by NSF Grant DMS-1438359. }
\maketitle

\begin{abstract}
For an embedded submanifold $\Sigma\subset\mathbb{R}^{N}$, Belkin and Niyogi showed that one can approximate the Laplacian operator using
heat kernels. Using a definition of coarse Ricci curvature derived by
iterating Laplacians, we approximate the coarse Ricci curvature of
submanifolds $\Sigma$ in the same way. For this purpose, we derive asymptotics for the approximation of the Ricci curvature proposed in \cite{AW}. Specifically, we prove Proposition 3.2 in \cite{AW}.  

\end{abstract}
\tableofcontents


\section{Introduction}

In \cite{BN08}, Belkin and Niyogi show that the graph Laplacian of a point
cloud of data samples taken from an embedded submanifold in Euclidean space converges to
the Laplace-Beltrami operator of the induced metric on the underlying manifold assuming that the samples are uniformly distributed on the submanifold (See also
\cite{Heinetal}).  In \cite{AW}, the authors proved that under similar assumptions, it is possible to formulate a sample version of the Ricci curvature that converges to the Ricci curvature with respect to the induced metric of the submanifold. In other words, assuming that points are uniformly distributed on the submanifold, one can construct an estimator of the Ricci curvature on the graph determined by the sample points that converges in probability to the Ricci curvature of the induced metric. 
\\

The process for approximating the Ricci curvature from a point cloud representing the submanifold follows two steps. The first step consists of approximating the Ricci curvature by a 1-parameter family of operators adopting an approach that is reminiscent to the approximation of the Laplace-Beltrami operator on submanifolds of Euclidean space. More precisely, the approximation relies on using a kernel that resembles the fundamental solution of the heat equation in Euclidean space in the sense that the kernel is computed in terms of an integrand that employs only ambient space information  (even though the integrand will eventually be restricted to the submanifold for the purpose of defining a convolution operator on the submanifold). The second step consists of showing that under an adequate choice of scale, the 1-parameter family of operators that are used to approximate the Ricci curvature can in turn be approximated by their sample counterparts. Even though this last statement seems uninteresting and in a certain sense expected, the non-linearity of the operators employed to approximate the Ricci curvature complicates the convergence analysis so much so that the use of ordinary laws of large numbers is simply hopeless, and it becomes necessary to use instead \emph{uniform laws of large numbers} to establish a convergence result. These issues are addressed in \cite{AW}. The goal of this paper is to establish a detailed analysis of the first half of the approximation process for the Ricci curvature, that is, defining a 1-parameter family of operators that approximate the Ricci curvature. More specifically, in this article we will provide a detailed proof of Proposition 3.2 in \cite{AW} which is instrumental in our manifold learning program for estimating Ricci curvature. We provide the reader with a statement of an improved version of Proposition 3.2 in \cite{AW} in a proposition that we have labeled as Proposition 3.1 (see Proposition \ref{tconvergence} in Section \ref{bias}). Proposition \ref{tconvergence} presents us with a method for computing the coefficients in the expansion of our 1-parameter approximation of the Ricci curvature, but Theorem \ref{extrinsic bias} gives us a more general result if all we want is merely to prove convergence as opposed to computing all coefficients in the expansion. 
\\

As mentioned before, the process of defining a heat-kernel type approximation of the Ricci curvature seems to be a simple generalization of the analogous process for the Laplace-Beltrami operator, but as we will see in this article, even in this context which does not involve any uniform laws of large numbers (as opposed to \cite{AW}) there are some technical hurdles that make this approximation far from a routine generalization of the Laplace-Beltrami case in \cite{BN08}. 
A recurring theme in \cite{AW} and \cite{AW1} is a discussion centered around the \emph{Carr\'{e} du Champ}, which is an object that measures how far is the infinitesimal generator of a semi-group from satisfying the Leibnitz rule (see \eqref{cdc}). The Carr\'{e} du Champ applied to the Laplace-Beltrami operator results in a bi-linear form that captures a cross-term that in principle does not seem to be very interesting in itself, however, the observation that the Carr\'{e} du Champ defines a bi-linear form allows us to formulate \emph{successive iterations} of this process (see \eqref{cdci}), and in particular the second iteration of the Carr\'{e} du Champ involves the \emph{Bochner formula} in a way that makes a connection to the Ricci curvature possible. In \cite{AW}, the authors proposed approximations of the Carr\'{e} du Champ and its second iterate based on the heat-kernel approximation of the Laplace-Beltrami operator in \cite{BN08} that would in principle allow for a satisfactory approximation of the Ricci curvature. In order to derive asymptotics that justify these approximations, it is necessary to consider Taylor-like expansions for the test functions used to evaluate all the operators that are of interest to this article so that the ``homogeneous" components of these Taylor approximations evaluated at test function determine the target operator (e.g. Laplace-Beltrami, Carr\'{e} du Champ or any of its iterates), and the remainder term in the Taylor expansion determines the rate of convergence of the approximation to the target operator. This last outline is perfectly well suited for the linear case (i.e. Laplace-Beltrami operator), but unfortunately is not sufficient for the non-linear case (Carr\'{e} du Champ and its iterates) due not only to the non-linearity of the problem itself but also to scaling considerations. To be more precise about what we mean by scaling considerations, we remind the reader that the parameter used in the approximation of the Ricci curvature through integral operators (we call it $t$), represents a scale parameter that is used to expand test functions near a point and capture information about their derivatives and about the local geometry of the submanifold. Since we are using a Gaussian-like convolution kernel, the scaling that we need to consider is simply the scaling that would normally be used for the fundamental solution of the heat equation. On the other hand, our approximation process follows a progression, i.e. we approximate first the Laplace-Beltrami operator, followed by its Carr\'{e} du Champ and finally its iterate, and each of these iterations requires dividing by the full scale parameter $t$ whereas the natural rescaling involved in the analysis of the heat equation requires multiplying the space variable by $\sqrt{t}$. The reason why it is necessary to divide by the parameter $t$ in every step of the progression, is that the approximation the Laplace-Beltrami operator in \cite{BN08} and \cite{AW} by means of a 1-parameter family of integral is analogous to the approximation of the Laplace operator in Euclidean space via difference quotients of the heat kernel.       
\\

As we will see, the complications derived from the scaling used in our problem suggest that one may potentially require a very high degree of differentiability from the test functions involved in the approximation, or perhaps an almost intractable expansion in order to obtain a satisfactory convergence result. Note that in principle the fact that we may require an elevated number of derivatives from the functions used to evaluate our approximating operators may not seem to be particularly inconvenient, however, we remind the reader that the approximation result in \cite[Proposition 3.2]{AW} requires derivatives of order at most $5$ from the test functions used in our approximation, and the convergence result in Theorem \ref{extrinsic bias} requires derivatives of order at most 4. We mitigate these concerns by using two important observations. The first observation is that as it usually happens with Gaussian convolution kernels in Euclidean space, there is an \emph{odd-cancellation} property that makes the main expansions in this article tractable. Even though the odd cancellation property is an Euclidean property (in the sense that it makes use of the symmetry of $\R^{n}$ for $n$ arbitrary), the fact that we are using sufficiently small scales provides an approximate odd-cancellation property for embedded submanifolds of Euclidean space that is sufficient for for the purposes of this article. The second and more important observation is that the non-linearity of the Carr\'{e} du Champ,  far from complicating the convergence to the Ricci curvature, actually helps us to simplify the convergence analysis by providing a faster than expected decay rate for remainder terms, and this can be used to compensate for the disparity in the scaling that results from successive applications of operators that carry a division by the scale factor $t$. A consequence of this analysis will be that we are able to extend our method to sampling distributions beyond the uniform distribution, more specifically to distributions that are absolutely continuous with respect to the volume for of the underlying embedded submanifold as long as the density functions of these distributions are sufficiently regular.  
\\



We now outline the structure of the paper: Section \ref{prelim} summarizes the background necessary to define an approximation to the Ricci curvature as in \cite{AW} (see Section \ref{approximation}), in particular it contains a precise definition of the Carr\'{e} du Champ (see the discussion surrounding \eqref{cdc} and \eqref{cdci}). Section \ref{prelim} also contains a brief introduction to the notion of Coarse Ricci Curvature as defined \cite{AW1} (see Section \ref{coarse}), as well as a discussion on the extension of the approximations in Section \ref{approximation} to the case where points are sampled not necessarily uniformly from the underlying submanifold. Section \ref{bias} contains the core of the technical work of this paper, and is in this section where we exploit the the benefits of using the odd-cancellation property. We compare the results in this section with Proposition 3.2 in \cite{AW}. Section \ref{bias} culminates with a proof of Theorem \ref{extrinsic bias} which is done through Corollary \ref{Cor4}. Section \ref{miscellaneous} is a discussion of how to extend the results in Section \ref{bias} to the case of a non-uniform density.  Section \ref{riccisubmanifold} contains formulas relating higher derivatives of extrinsically ``linear'' functions to the second fundamental form, allowing us to show the convergence of the coarse Ricci curvature, connecting the results in this paper to those in \cite{AW1}. Finally, in Appendix \ref{appendixA} we formulate the odd-cancellation property used in Section \ref{bias}, in fact the main statement of odd-cancellation is given in Lemma \eqref{local_moments} and in Appendix \ref{appendixB} we give detailed proofs of some of the most laborious estimates in Section \ref{bias}. 

\section*{Acknowledgements} The first author was partially supported by a postdoctoral fellowship of the National Science Foundation, award No. DMS-1204742. The second author was  partially supported by NSF Grant DMS-1438359. The authors are grateful to the anonymous referees who provided input and comments that greatly helped to improve the presentation of this article.

\section{Preliminaries and statement of results}\label{prelim}

\subsection{Iterated Carr\'{e} du Champ}
Let $(X,\mathcal{B})$ be a measure space and let $P_{t}$ be a 1-parameter family of operators of the form
\begin{align*}
    P_{t}f(x) =\int_{X}f(y)p_{t}(x,dy),
\end{align*}
for $x$ in $X$. (Note that it will be sufficient for the scope of our paper that this operator is defined on bounded functions.  See also remark \ref{compactspace}.)   We assume that $\{P_{t}\}_{t\ge 0}$ is a \emph{semi-group} in the sense that 
\begin{align*}
    P_{s}\circ P_{t} =P_{t+s}\\
    P_{0} = \mathrm{Id}.
\end{align*}
Let $L$ be the infinitesimal generator of a semi-group $\{P_{t}\}_{t\ge 0}$, which formally means that $L$ is defined by
\begin{align*}
    Lf=\frac{d}{dt}P_{t}f|_{t=0}=\lim_{t\rightarrow 0^{+}}\left(\frac{P_{t}f-f}{t}\right).
\end{align*}
The \emph{Carr\'{e} du Champ} of $L$ is defined as an operator that measures how far is the generator $L$ from being a derivation, i.e., how far is $L$ from satisfying the Leibnitz rule. More precisely, the Carr\'{e} du Champ of $L$ evaluated at functions $u,v$ is defined by the bilinear form 
\begin{equation}
\Gamma(L,u,v)=\frac{1}{2}\left(  L(uv)-L(u)v-uL(v)\right).  \label{cdc}%
\end{equation}
We will also consider the \emph{iterated Carr\'{e} du Champ} introduced by
Bakry and \'{E}mery \cite{BE85} denoted by $\Gamma_{2}$ and defined by
\begin{equation}
\Gamma_{2}(L,u,v)=\frac{1}{2}\left(  L(\Gamma(L,u,v))-\Gamma(L,Lu,v)-\Gamma
(L,u,Lv)\right)  . \label{cdci}%
\end{equation}

When $L$ is the rough Laplacian with respect to the metric $g$, then
\[
\Gamma(\Delta_{g},u,v)=\langle\nabla u,\nabla v\rangle_{g}.
\]
A simple computation shows that in this case we have the following identity for the iterated Carr\'{e} du Champ of the Laplacian 
\begin{align*}
    \Gamma_{2}(\Delta_{g},u,v) = \frac{1}{2}\Delta_{g}\langle\nabla u,\nabla v\rangle_{g}-\frac{1}{2}\langle\nabla\Delta_{g}u,\nabla v\rangle_{g}-\frac{1}{2}\langle\nabla v,\nabla \Delta_{g}v\rangle_{g}
\end{align*}
and commuting covariant derivatives we arrive at the \emph{Bochner formula} 
\begin{align}
    \Gamma_{2}(\Delta_{g},u,v) = \mathrm{Ric}_{g}(\nabla u,\nabla v)+\left\langle\nabla^{2}_{g}u,\nabla^{2}_{g}v\right\rangle_{g},\label{Boch}
\end{align}
where of course $\mathrm{Ric}_{g}$ is the Ricci curvature of $g$. Identity \eqref{Boch} is the key identity for recovering the Ricci curvature, and the details of this process are explained in \cite{AW}, however the reader may find Section \ref{riccisubmanifold} useful to understand some of the key aspects of the approximation of Ricci.

\begin{notation}
\emph{{ When considering the operators \eqref{cdc} and \eqref{cdci} we will
use the slightly cumbersome three-parameter notation, as the main results will
be stated in terms of a family of operators $\{L_{t}\}.$ } }
\end{notation}

\subsection{Coarse Ricci Curvature}\label{coarse}

In this section we provide a definition of coarse Ricci curvature on general
metric measures spaces, using a family of operators which are intended to
approximate a Laplace operator on a metric space at scale $t.$ The coarse Ricci
curvature will then be defined on pairs of points. To obtain a quantity for
the operator (\ref{cdci}), we need a function to evaluate. For submanifolds in
Euclidean space, we use a linear function whose gradient is the vector that
points from a point $x$ to a point $y$. On a general metric space $(X,d)$, given
$x,y\in X$ define
\[
f_{x,y}(z)=\frac{1}{2}\left(  d^{2}(x,y)-d^{2}(y,z)+d^{2}(z,x)\right)  .
\]
Note that if $X=\R^{N}$ and $d(x,y)=\|x-y\|_{N}$ is the Euclidean distance in $\R^{N}$ then 
\begin{equation}
f_{x,y}(z)=\langle y-x,z\rangle,\label{fxy}%
\end{equation}
where $\langle\cdot,\cdot\rangle$ is the Euclidean inner product in $\R^{N}$. This leads us to the following definition of coarse Ricci curvature,

\begin{definition}
Given an operator $L$ we define the coarse Ricci curvature for $L$ as
\begin{align}
\mathrm{Ric}_{L}(x,y)=\Gamma_{2}(L,f_{x,y},f_{x,y})(x),\label{RicL}
\end{align}

\end{definition}

Apart from the connection that we established above in \eqref{Boch} between $\Gamma_2$ and the Ricci curvature, there is a more important observation derived by Bakry and \'Emery in \cite{BE85} regarding the Carr\'{e} du Champ, that is, the bilinear form $\Gamma_2$ can be used to formulate properties of Ricci curvature lower bounds. Given an operator $L$ as before, we define a curvature-dimension condition on a smooth measure metric space $(X,g,\nu)$ in the following way: we say that $L$ satisfies the \emph{$CD(k,\mathcal{N})$
condition} if there exist measurable functions $k:X\rightarrow\R$ and $\mathcal{N}:X\rightarrow[1,\infty]$ such that for every $f$ defined on a set of functions dense in $L^{2}(X,\nu)$ the inequality
\begin{align}
    \Gamma_{2}(L,f,f)\ge\frac{1}{\mathcal{N}}(Lf)^2+k\Gamma(L,f,f)\label{CDKN}
\end{align}
holds. In this definition $k$ stands for curvature and $\mathcal{N}$ for dimension. On the other hand, when the measure $d\nu$ is given by $e^{-\vartheta}d\mathrm{vol}$ where $d\mathrm{vol}$ is the volume form corresponding to the metric $g$,  $\vartheta$ is a smooth function, and $X$ has dimension $n$, Bakry and \'Emery arrived at the following dimension and weight-dependent definition of Ricci curvature
\begin{align}
    \mathrm{Ric}_{\mathcal{N}}=\left\{
    \begin{array}{ll}\label{RicNN}
    \mathrm{Ric}_g+\mathrm{Hess}_{\vartheta}&\text{if}~\mathcal{N}=\infty\\
    \mathrm{Ric}_g+\mathrm{Hess}_{\vartheta}-\frac{1}{\mathcal{N}-n}(d\vartheta\otimes d\vartheta)&\text{if}~n<\mathcal{N}<\infty\\
    \mathrm{Ric}_g+\mathrm{Hess}_{\vartheta}-\infty(d\vartheta\otimes d\vartheta)&\text{if}~\mathcal{N}=n\\
    -\infty & \text{if}~ \mathcal{N}<n.
    \end{array}
    \right.
\end{align}
Bakry and \'Emery showed in \cite{BE85} the equivalence between the $CD(k,\mathcal{N})$ condition in \eqref{CDKN} and the bound $\mathrm{Ric}_{\mathcal{N}}\ge k$. We recall the main results from \cite{AW1}.

\begin{theorem}
\bigskip Consider a smooth measure metric space $(X^{n},g,e^{-\vartheta}d\mathrm{vol})$ and let $\Delta_{\vartheta}$ be the operator defined by 
\begin{align}
\Delta_{\vartheta}f=\Delta_{g}f-\langle\nabla \vartheta,\nabla f\rangle_{g}.
\end{align}
The operator $\Delta_{\vartheta}$ is called the weighted Laplacian for the measure $e^{-\vartheta}d\mathrm{vol}$, $\vartheta$ is smooth and $d\mathrm{vol}$ is the volume metric corresponding to the metric $g$. Let as before
\begin{align}
\mathrm{Ric}_{\infty}=\mathrm{Ric}_{g}+\mathrm{Hess}_{\vartheta}=\mathrm{Ric}_{g}+\nabla^{2}_{g}\vartheta.
\end{align}
Then
\begin{equation}
\mathrm{Ric}_{\infty}(\gamma^{\prime}\left(  0\right)  ,\gamma^{\prime}\left(
0\right)  )=\frac{1}{2}\frac{d^{2}}{ds^{2}}\mathrm{Ric}_{\Delta_{\vartheta}}%
(x,\gamma\left(  s\right)  )\bigg\rvert_{s=0}, \label{secondder}%
\end{equation}
where $\mathrm{Ric}_{\Delta_{\vartheta}}$ is defined in the sense of \eqref{RicL}
and%
\begin{equation}
\mathrm{Ric}_{\infty}\geq k
\end{equation}
if and only if%
\[
\mathrm{Ric}_{\Delta_{\vartheta}}(x,y)\geq kd^{2}(x,y).
\]

\end{theorem}

As mentioned in the introduction, the ambient distance squared function
osculates the intrinsic distance squared function only to third order on the
diagonal along the submanifold. So the above formula could manifest some
error terms. To side-step this, we appeal to the Bochner formula, which
in this case can be stated formally as
\[
\Gamma_{2}(\Delta_{g},f,f)=\mathrm{Ric}_{g}(\nabla f,\nabla f)+\Vert\nabla
^{2}_{g}f\Vert_{g}^{2}.
\]
We note that if we evaluate $\Gamma_{2}$ on functions with vanishing Hessian
at a point, we can recover the Ricci curvature exactly. For submanifolds of
Euclidean space $\R^{N}$, we normalize the functions (\ref{fxy}) to linear function
whose gradient is the \textit{unit} vector that points from a point $x$ to a
point $y$. \ In particular, given $x,y$
\[
F_{x,y}(z)=\frac{1}{2}\frac{\|x-y\|^{2}_{N}-\|y-z\|^{2}_{N}-\|x-y\|_{N}^{2}}{\|z-x\|_{N}},
\]
that is%
\begin{align}\label{FXY}
F_{x,y}(z)=\left\langle\frac{y-x}{\left\| y-x\right\|_{N} },z\right\rangle.
\end{align}
This leads us to the following definition of \textit{life-sized} coarse Ricci curvature.

\begin{definition}
Given an operator $L$ we define the life-sized coarse Ricci curvature for $L$
as
\[
\mathrm{RIC}_{L}(x,y)=\Gamma_{2}(L,F_{x,y},F_{x,y})(x).
\]

\end{definition}

As we will see, this also can be used to recover the Ricci curvature, without
taking any derivatives. \ \newline

\subsubsection{Approximations of the Laplacian, Carr\'{e} du Champ and its
iterate.}\label{approximation}

We now construct operators which can be thought of as approximations of the
Laplacian on metric measure spaces. This construction is a slight modification
of the approximation constructed by Belkin-Niyogi in \cite{BN08} and more
generally Coifman-Lafon in \cite{CoLaf06}. Consider a metric measure space
$(X,d,\nu)$ with the Borel $\sigma$-algebra such that $\nu(X)<\infty$. Given
$t>0$, let $\theta_{t}$ be given by
\begin{equation}
\theta_{t}(x)=\int_{X}e^{-\frac{d^{2}(x,y)}{2t}}d\nu(y).
\label{generaldensity}%
\end{equation}
We define a 1-parameter family of operators $L_{t}$ as follows: given a
function $f$ on $X$ define
\begin{equation}
L_{t}f(x)=\frac{2}{t\theta_{t}(x)}\int_{X}\left(  f(y)-f(x)\right)
e^{-\frac{d^{2}(x,y)}{2t}}d\nu(y), \label{Ltdefine}%
\end{equation}
for $f\in L^{2}(X,\nu)$. With respect to the operator $L_{t}$ one can define a Carr\'{e} du Champ on
$L^{2}(X,\nu)$ functions $f,h$ by
\begin{equation}
\Gamma(L_{t},f,h)=\frac{1}{2}\left(  L_{t}(fh)-(L_{t}f)h-f(L_{t}h)\right)  ,
\label{cdctdef}%
\end{equation}
which simplifies to
\begin{equation}
\Gamma(L_{t},f,h)(x)=\frac{1}{t\theta_{t}(x)}\int_{X}e^{-\frac{d^{2}(x,y)}%
{2t}}(f(y)-f(x))(h(y)-h(x))d\nu. \label{cdctsimp}%
\end{equation}
In a similar fashion we define the iterated Carr\'{e} du Champ of $L_{t}$ to
be
\begin{equation}
\Gamma_{2}(L_{t},f,h)=\frac{1}{2}\left(  L_{t}(\Gamma(L_{t},f,h))-\Gamma
(L_{t},L_{t}f,h)-\Gamma(L_{t},f,L_{t}h)\right)  . \label{cdc2t}%
\end{equation}
In this article we are interested mainly in the case $X=\Sigma^{d}$, where $\Sigma$ is a $d$-dimensional \emph{smooth} embedded submanifold of $\R^{N}$ and the metric $d$ is given by $d(x,y)=\|x-y\|_{N}$ where as before $\|\cdot\|_{N}$ is the Euclidean norm in $\R^{N}$. In other words, the metric that we will choose in $\Sigma$ is the restriction of the ambient metric to the submanifold $\Sigma^{d}$. In this case, even though $d(x,y)$ is given by the ambient distance, the geometry of $\Sigma$ is determined by the \emph{induced metric}, i.e. the metric $g_{\Sigma}$ induced by the embedding $F:\Sigma^{d}\rightarrow\R^{N}$. The operator $L_{t}f$ takes the form
\begin{align*}
    L_{t}f(x) =\frac{2}{t\cdot\theta_{t}(x)}\int_{\Sigma^{d}}e^{-\frac{\|x-y\|^2_{N}}{2t}}(f(y)-f(x))d\mu(y),
\end{align*}
where $d\mu$ is the volume form $\mathrm{vol}_{\Sigma}$ defined by the metric $g_{\Sigma}$, and $\theta_{t}(x)$ is given by 
\begin{align*}
    \theta_{t}(x)=\int_{\Sigma^{d}}e^{-\frac{\|x-y\|^{2}_{N}}{2t}}d\mu(y).
\end{align*}
The operators $\Gamma(L_t,f,h)$ and $\Gamma_{2}(L_{t},f,h)$ are defined accordingly, and observe in particular that 
\begin{align*}
 \Gamma(L_t,f,h)=\frac{1}{t\cdot\theta_{t}(x)}\int_{\Sigma^{d}}e^{-\frac{\|x-y\|^{2}_{N}}{2t}}(f(y)-f(x))(h(y)-h(x))d\mu(y).   
\end{align*}
\begin{remark} We make note of the following: 

\begin{itemize}

\item This definition of $L_{t}$ differs from the operator introduced by Belkin-Niyogi in \cite{BN08}, in that we
normalize by dividing by $\theta_{t}(x)$ instead of dividing by $(2\pi t)^{d/2}$ for an assumed
manifold dimension $d.$
\item For the case in question, i.e. the case when $X=\Sigma^{d}$ for $\Sigma^{d}$ a closed embedded submanifold of $\R^{N}$, $\nu=d\mu$ and $d(x,y)=\|x-y\|_{N}$ the operator $L_{t}$ maps $L^{2}(\Sigma,d\mu)$ to itself.  
\end{itemize}
\end{remark}

\subsection{Statement of Results}

We will consider a closed, smooth, embedded submanifold $\Sigma$ of
$\mathbb{R}^{N}$, and the metric measure space will be $(\Sigma^{d},\Vert
\cdot\Vert_{N},d\mathrm{vol})$, where

\begin{itemize}
\item $\|\cdot\|_{N}$ is the distance function in the ambient space $\mathbb{R}%
^{N}$,

\item $d\mathrm{vol}_{\Sigma}$ is the volume element corresponding to the
metric $g$ induced by the embedding of $\Sigma$ into $\mathbb{R}^{N}$.
\end{itemize}

In addition we will adopt the following conventions

\begin{itemize}
\item All operators $L_{t}$, $\Gamma(L_{t},\cdot,\cdot)$ and $\Gamma_{2}%
(L_{t},\cdot,\cdot)$ will be taken with respect to the distance $\Vert
\cdot\Vert_{N}$ and the measure $d\mathrm{vol}_{\Sigma}$.
\end{itemize}

The choice of the above metric measure space is consistent with the setting of
manifold learning in which no assumption on the geometry of the submanifold
$\Sigma$ is made, in particular, we have no a priori knowledge of the geodesic
distance and therefore we can only hope to use the chordal distance as a
reasonable approximation for the geodesic distance. We will show that while
our construction at a scale $t$ involves only information from the ambient
space, the limit as $t$ tends to $0$ will recover the life-size coarse Ricci
curvature of the submanifold with intrinsic geodesic distance. As pointed out
by Belkin-Niyogi \cite[Lemma 4.3]{BN08}, the chordal and intrinsic distance
functions on a smooth submanifold differ first at fourth order near a point ,
so while much of the analysis is done on submanifolds, the intrinsic geometry
will be recovered in the limit. We are able to show the following.

\begin{theorem}
\label{extrinsic bias} Let $\Sigma^{d}\subset\mathbb{R}^{N}$ be a closed
embedded submanifold, let $g$ be the Riemannian metric induced by the
embedding, and let $(\Sigma,\Vert\cdot\Vert_{N},d\mathrm{vol}_{\Sigma})$ be the
metric measure space defined with respect to the ambient distance. Suppose
that $f\in C^{4}\left(  \Sigma\right)  .$
Then there exists a constant $C_{1}$ depending on the geometry of $\Sigma$ and
the derivatives of the function $f$ up to order $4$ such that
\begin{align}
\sup_{x\in\Sigma}\left\vert \Gamma_{2}(\Delta_{g},f,f)(x)-\Gamma_{2}%
(L_{t},f,f)(x)\right\vert <C_{1}(\Sigma,\|f\|_{C^{4}(\Sigma)})\cdot t^{1/2}.\label{MainInequality}
\end{align}

\end{theorem}

Theorem \ref{extrinsic bias} will follow from Lemmas \ref{l319} and \ref{FDR}
which are proved in Section \ref{bias} (see also the remarks at the end of Section \ref{bias} where the proof of Theorem \ref{extrinsic bias} is given formally).

\begin{corollary}
\label{ricdelta} With the hypotheses of Theorem \ref{extrinsic bias} we have
\[
\mathrm{Ric}_{\Delta_{g}}(x,y)=\lim_{t\rightarrow0}\Gamma_{2}(L_{t}%
,f_{x,y},f_{x,y})(x).
\]

\end{corollary}

Theorem \ref{extrinsic bias} applies to all functions on the manifold. \ To
obtain the life-size Ricci curvature we apply these to $F_{x,y}$ to obtain the
following. \ 

\begin{theorem}
\label{HereWeRecoverRicci} Let $\Sigma^{d}\subset\mathbb{R}^{N}$ be a closed
embedded submanifold, and let $g$ be the metric induced by the embedding. Let
$\gamma(s)$ be a unit speed geodesic in $\Sigma$ such that $\gamma(0)=x.$
There exist constants $C_{2},C_{3}$ depending on the geometry of $\Sigma$, the embedding coordinates of $\Sigma$ and bounds on the second derivatives of the second fundamental form of $\Sigma$ 
such that
\[
|\mathrm{Ric}_{g}(\gamma^{\prime}(0),\gamma^{\prime}(0))-\mathrm{RIC}_{L_{t}%
}(x,\gamma(s))|\leq C_{2}t^{1/2}+C_{3}s.
\]

\end{theorem}

This will be proved in section \ref{riccisubmanifold}.

\subsubsection{Smooth Metric Measure Spaces and non-Uniformly Distributed
Samples}\label{SMMS}

Consider a smooth metric measure space $\left(  M,g,e^{-\vartheta}d\mathrm{vol}%
\right)  $ and let $\Delta_{\vartheta}$ be the operator
\[
\triangle_{\vartheta}u=\triangle_{g}u-\langle\nabla \vartheta,\nabla u\rangle_{g}.
\]
In \cite{CoLaf06}, the authors consider a family of operators $L_{t}^{\alpha}$
which converge to $\triangle_{2(1-\alpha)\vartheta}.$ Note that a standard computation
(cf. \cite[Page 384]{vill09}) gives
\begin{equation}
\Gamma_{2}(\triangle_{2(1-\alpha)\vartheta},f,f)=\frac{1}{2}\Delta_{g}\left\Vert
\nabla f\right\Vert _{g}^{2}-\langle\nabla \vartheta,\nabla\Delta_{g}f\rangle
_{g}+2(1-\alpha)\nabla_{g}^{2}\vartheta(\nabla f,\nabla f).\label{mm18}%
\end{equation}
We adapt \cite{CoLaf06} to the setting of embedded submanifolds of Euclidean space: Recall that
\begin{equation}
\theta_{t}(x)=\int_{\Sigma^{d}}e^{-\frac{\|x-y\|^{2}_{N}}{2t}}d\nu(y)\label{generaldensityalpha},
\end{equation}
where $d\nu$ is equivalent to $e^{-\vartheta}d\mathrm{vol}_{\Sigma}=e^{-\vartheta}d\mu$, and define, for $\alpha\in\lbrack0,1]$
\begin{equation}
\theta_{t,\alpha}(x)=\int_{\Sigma^{d}}e^{-\frac{\|x-y\|_{N}^2}{2t}}\frac{1}{\left[
\theta_{t}(y)\right]  ^{\alpha}}d\nu(y).\label{alpha1}%
\end{equation}
We can define the operator
\begin{equation}
L_{t}^{\alpha}f(x)=\frac{2}{t}\frac{1}{\theta_{t,\alpha}(x)}\int_{\Sigma^{d}}
e^{-\frac{\|x-y\|_{N}^2}{2t}}\frac{1}{\left[  \theta_{t}(y)\right]  ^{\alpha}%
}\left(  f(y)-f(x)\right)  d\nu(y)\label{alpha2},
\end{equation}
and again obtain bilinear forms $\Gamma(L_{t}^{\alpha},f,f)$ and $\Gamma
_{2}(L_{t}^{\alpha},f,f)$. We consider the metric measure space $(\Sigma^{d}
,\Vert\cdot\Vert_{N},e^{-\vartheta}d\mathrm{vol}_{\Sigma})$ where $\Sigma^{d}%
\subset\mathbb{R}^{N}$ is an embedded submanifold, $\Vert\cdot\Vert_{N}$ is the
ambient distance and $\varrho$ is a smooth function on $\Sigma$. We again take
all the operators $L_{t},\Gamma_{t}(L_{t},\cdot,\cdot)$ and $\Gamma_{2}%
(L_{t},\cdot,\cdot)$ with respect to the data of $(\Sigma,\Vert\cdot
\Vert_{N},e^{-\vartheta}d\mathrm{vol}_{\Sigma})$.

\begin{theorem}
\label{intrinsic bias} Let $\Sigma^{d}\subset\mathbb{R}^{N}$ be an embedded
submanifold and consider the smooth metric measure space $(\Sigma,\Vert
\cdot\Vert_{N},e^{-\vartheta}d\mathrm{vol}_{\Sigma})$. Let $f\in C^{4}(\Sigma)$ such that
$\Vert f\Vert_{C^{4}(\Sigma)}\leq M$. Let $g$ be the metric induced by the embedding of $\Sigma^{d}$ into $\R^{N}$. There exists a constant $C_{4}$ depending on the geometry of $\Sigma$, the embedding coordinates and bounds on $\|f\|_{C^{4}(\Sigma)}$ such that%
\begin{align*}
\sup_{\xi\in\Sigma}\left\vert \Gamma_{2}(L_{t}^{\alpha},f,f)(\xi)-\Gamma
_{2}(\triangle_{2(1-\alpha)\vartheta},f,f)(\xi)\right\vert \leq C_{4}(\Sigma,\|f\|_{C^{4}(\Sigma)})t^{1/2}.
\end{align*}

\end{theorem}

In particular, if the density function is positive and smooth, we can still
recover the Ricci curvature using $\alpha=1$ in the above expression. The proof of Theorem \ref{intrinsic bias} is similar (although much more
tedious) than the proof of Theorem \ref{extrinsic bias} and will be outlined in Section \ref{miscellaneous}. 

\section{Bias Error Estimates}\label{bias}

\subsection{Bias for Submanifold of Euclidean Space}

In this section we prove Theorem \ref{extrinsic bias}. The theorem will follow
from Lemmas \ref{l319} and \ref{FDR}, by (\ref{cdc2t}). Theorem \ref{extrinsic bias} has an alternative formulation in \cite[Proposition 3.2]{AW}, and as mentioned in the introduction, the proof of this proposition was deferred to this article. We now present an improved version of  \cite[Proposition 3.2]{AW} which of course implies \cite[Proposition 3.2]{AW}.  For simplicity we
will assume that $(\Sigma,d\mathrm{vol}_{\Sigma})$ has unit volume. We will also use $d\mu$ to denote the volume form $d\mathrm{vol}_{\Sigma}$. Recall the
definitions (\ref{generaldensity}), (\ref{Ltdefine}), \eqref{cdctdef},
\eqref{cdctsimp} and \eqref{cdc2t}.
\begin{proposition}[See Proposition 3.2 in \cite{AW}]\label{tconvergence} 
Suppose that
$\Sigma^{d}$ is a closed, embedded, unit volume submanifold of $\mathbb{R}%
^{N}$. Let also $g$ be the metric induced by the embedding of $\Sigma^{d}$
into $\mathbb{R}^{N}$. For any $x$ in $\Sigma$ and for any functions $f,h$ of class 
$C^{5}(\Sigma)$ we have
\begin{align}
\frac{(2\pi t)^{d/2}}{\theta_{t}(x)}  &  =1+tG_{1}(x)+R_{2}(t,x)%
(x),\label{densityt}\\
\Gamma(L_{t},f,h)(x)  &  =\langle\nabla_{\Sigma} f(x),\nabla_{\Sigma}
h(x)\rangle_{g}\label{cdct}\\
&  +tG_{1}(x,J^{3}(f)(x),J^{3}(h)(x))+R_{2}(t,x,J^{4}f(x),J^{4}%
h(x)),\nonumber\\
L_{t}f(x)  &  =\Delta_{g}f(x)+tG_{1}%
(x,J^{4}f(x))+R_{2}(t,x,J^{5}f(x)), \label{Ltexpand}%
\end{align}
and
\begin{equation}
\Gamma_{2}(L_{t},f,f)(x)=\Gamma_{2}(\Delta_{g},f,f)(x)+t^{1/2}R%
(t,x,J^{5}f(x)) \label{gamma2approx}%
\end{equation}
where each $G_{i}$ is a locally defined function, which is smooth in its
arguments, and $J^{k}(u)$ is a locally defined $k$-jet of the function $u$.
Also, we have the following bounds for the remainder terms above:
\begin{align*}
    \sup_{x\in\Sigma}\left|R_{2}(t,x)\right|\le C_{\Sigma}\cdot t^{2},
\end{align*}
\begin{align*}
    \sup_{x\in\Sigma}\left|R_{2}(t,x,J^{4}f(x),J^{4}h(x))\right|\le C(\Sigma,\|f\|_{C^{4}(\Sigma)},\|h\|_{C^{4}(\Sigma)})\cdot t^2,
\end{align*}
\begin{align*}
    \sup_{x\in\Sigma}\left|R_{2}(t,x,J^{5}f(x))\right|\le C(\Sigma,\|f\|_{C^{4}(\Sigma)})\cdot t^2,
\end{align*}
\begin{align*}
    \sup_{x\in\Sigma}\left|R(t,x,J^{5}f(x))\right|\le C(\Sigma,\|f\|_{C^{5}(\Sigma)}),
\end{align*}
where each constant of the form
\begin{align}
C(\Sigma,\|f_1\|_{C^{k_1}(\Sigma)},\ldots,\|f_m\|_{C^{k_m}(\Sigma)})
\end{align}
depends on the embedding coordinates, the geometry of $\Sigma$, and bounds on the norms $\|f_1\|_{C^{k_1}(\Sigma)},\ldots,\|f_{m}\|_{C^{k_m}(\Sigma)}$ for tensors $f_1,\ldots,f_{m}$  as in \eqref{MainInequality}.
\end{proposition}

The proof of Proposition \ref{tconvergence} is given at the end of this section. We will devote the rest of this section to derive the estimates that are necessary to prove Proposition \ref{tconvergence}, however the actual proof of the proposition will be given towards the end of the section. The technical discussions in this section will also yield a proof of Theorem \ref{intrinsic bias}. The general strategy for proving Proposition \ref{tconvergence} will comprise two steps. The \emph{first step} will consist of determining a number $\varepsilon$
depending on the intrinsic and extrinsic geometry of $\Sigma$, such the
ambient and geodesic balls of radius $\varepsilon$ are comparable, and both have
trivial topology and in particular a choice of orthonormal tangent frame on
the ball. We then cover the compact manifold by a finite collection of balls
of radius $\varepsilon,$ which we call $\left\{  \Omega_{i}\right\}  $. The \emph{second step} will consist of an analysis on any $\Omega_{i}.$ The goal is to split the
quantities that go into the expression (\ref{generaldensity}), (\ref{Ltdefine}%
), etc., each into a local quantity determined inside $\Omega_{i}$ and an
error term that exponentially decays as $t$ vanishes. The statement of our
main results, Lemmas \ref{l319} and \ref{FDR}, are point-wise local statements
that hold on each $\Omega_{i}$. There is a finite set of $\Omega_{i}$ so one
can make the estimates global by taking worst error terms over the finite set
of $\Omega_{i}$.

\begin{remark} \label{compactspace} Note that we choose compact for simplicity, however, if we assume the manifold
is locally compact and has polynomial volume growth at infinity, the following
analysis should work when carefully applied locally on compact sets.
\end{remark}

\subsection{Finding sufficiently small balls to work on}

First we note that, by a covering argument there will be some $\varepsilon
_{0}>0$ such that all balls of radius $\varepsilon_{0}$ admit an orthonormal
frame for the tangent space. Let $\varepsilon_{1}$ be the injectivity radius
of the manifold, and let $\varepsilon_{2}$ be such that all ambient balls of
radius $2\varepsilon$ intersect the manifold in topological balls for each
$\varepsilon<\varepsilon_{2}.$

Now let $x$ lie at the center of a geodesic ball of radius less than
$\varepsilon_{1}$. In this case for all points $y$ such that $d_{\Sigma
}(x,y)<\varepsilon_{1}$ we may write $y=\exp_{x}(z)$ for $z$ in
$B_{\varepsilon_{1}}^{d}(0)$, i.e., the ball with radius $\varepsilon_{1}$
centered at $0$ in $\mathbb{R}^{d}$ (after identifying $T_{x}\Sigma$ with
$\mathbb{R}^{d}$). Letting $\Vert z\Vert_{d}$ denote the Euclidean norm in
$\mathbb{R}^{d}$, it follows that $d_{\Sigma}(x,y)=\Vert z\Vert_{d}$. On the
other hand, we clearly have $d_{\Sigma}(x,y)\geq\Vert x-y\Vert_{N}$ and we
therefore obtain the relation
\begin{equation}
\Vert z\Vert_{d}^{2}=\Vert x-\exp_{x}(z)\Vert_{N}^{2}+\phi(z),
\label{define_phi}%
\end{equation}
where $\phi(z)$ is a non-negative function that is smooth whenever $\Vert
z\Vert_{d}<\varepsilon_{1}.$ The following is found in \cite{BN08}.

\begin{lemma}
[\cite{BN08}]\label{lemGeo} Let $\Sigma^{d}$ be a smooth closed embedded
submanifold of $\mathbb{R}^{N}$ and let $x$ be a point in $\Sigma$. Let
$d_{\Sigma}$ denote the geodesic distance corresponding to the induced metric
on $\Sigma$. Then,
\begin{align*}
d^{2}_{\Sigma}(x,y)-\|x-y\|_{N}^{2}=O(d^{4}_{\Sigma}(x,y)).
\end{align*}
In particular there exists $C_{\Sigma}$ depending only on the embedding
coordinates of $\Sigma$ such that given $z$ in any exponential chart we have
\begin{equation}
\phi(z)\le C_{\Sigma}\|z\|^{4}_{d} \label{BNBound}%
\end{equation}

\end{lemma}

We define $\varepsilon_{3}$ in the following corollary:

\begin{corollary}
\label{exp_phi} Choose normal coordinates $z$ at a point $x_{0}$. \ There
exists $\varepsilon_{3}>0$ with $\varepsilon_{3}$ strictly less than the
injectivity radius of $\Sigma$ with respect to the induced metric such that
for any $t>0$ and for all $x$ in $\Sigma$ and all $y=$ $\exp_{x_{0}}(z)$ in
$B_{\varepsilon_{3}}(x)$ one has
\[
\left\vert e^{-\frac{\Vert x_{0}-\exp_{x_{0}}z\Vert_{N}^{2}}{2t}}%
-e^{-\frac{\Vert z\Vert_{d}^{2}}{2t}}\left(  1+\frac{\phi(z)}{2t}\right)
\right\vert \leq\frac{1}{2}\left(  \frac{\phi(z)}{2t}\right)  ^{2}%
e^{-\frac{\Vert z\Vert_{d}^{2}}{4t}}.
\]

\end{corollary}

\begin{proof}
We start by writing
\begin{align}
e^{\frac{\phi(z)}{2t}}-1  &  =\int_{0}^{1}\frac{d}{ds}\left(  e^{\frac
{s\phi(z)}{2t}}\right)  ds\nonumber\\
&  =\frac{\phi(z)}{2t}\int_{0}^{1}e^{\frac{s\phi(z)}{2t}}ds\nonumber\\
&  =\frac{\phi(z)}{2t}+\frac{\phi(z)}{2t}\int_{0}^{1}(1-s)\frac{d}{ds}%
e^{\frac{s\phi(z)}{2t}}ds\nonumber\\
&  =\frac{\phi(z)}{2t}+\left(  \frac{\phi(z)}{2t}\right)  ^{2}\int_{0}%
^{1}(1-s)e^{s\frac{\phi(z)}{2t}}ds. \label{Jordan P}%
\end{align}
It follows that
\begin{align*}
e^{-\frac{\Vert x-\exp_{x_{0}}z\Vert_{N}^{2}}{2t}}  &  =e^{-\frac{\Vert
z\Vert_{d}^{2}}{2t}}e^{\frac{\phi(z)}{2t}}\\
&  =e^{-\frac{\Vert z\Vert_{d}^{2}}{2t}}\left(  1+\frac{\phi(z)}{2t}\right)
+e^{-\frac{\Vert z\Vert_{d}^{2}}{2t}}\left(  \frac{\phi(z)}{2t}\right)
^{2}\int_{0}^{1}(1-s)e^{\frac{s\phi(z)}{2t}}ds\\
&  \leq e^{-\frac{\Vert z\Vert_{d}^{2}}{2t}}\left(  1+\frac{\phi(z)}%
{2t}\right)  +\frac{1}{2}\left(  \frac{\phi(z)}{2t}\right)  ^{2}%
e^{-\frac{\Vert z\Vert_{d}^{2}}{2t}+\frac{\phi(z)}{2t}}%
\end{align*}
and since for $\Vert z\Vert_{d}^{2}\leq\varepsilon_{1}$ we have $|\phi(z)|\leq
C_{\Sigma}\Vert z\Vert_{d}^{4}$ for $C_{\Sigma}>0$ depending only on the
embedding coordinates of $\Sigma,$ we can choose $\varepsilon_{3}>0$
sufficiently small so that
\[
e^{-\frac{\Vert z\Vert_{d}^{2}}{2t}+\frac{\phi(z)}{2t}}\leq e^{-\frac{\Vert
z\Vert_{d}^{2}}{4t}},
\]
and this proves the lemma.
\end{proof}

The following notation will be used for the remainder of the article.

\begin{notation}
We have the following

\begin{itemize}
\item $B_{r}^{d}(0)$ denotes a ball centered at the origin in $\mathbb{R}^{d}$
\ (which is identified with the tangent space at any point.)\ 
\item Recall that we have used $\Vert\cdot\Vert_{d}$ to denote the Euclidean norm in $\mathbb{R}^{d}.$
\item We will use $g_{\Sigma}$ to denote the induced metric of $\Sigma$.
\end{itemize}
\end{notation}

At this point we can take a cover of the manifold by balls of radius
\begin{equation}
\varepsilon=\frac{1}{2}\min\{\varepsilon_{0},\varepsilon_{1},\varepsilon
_{2},\varepsilon_{3}\}. \label{epschoice}%
\end{equation}
A quantity that will often appear in the estimates to come is the following
number
\begin{equation}
d_{\varepsilon}=\inf_{x\in\Sigma}\left\{  \inf_{y\in\Sigma\setminus
B_{\varepsilon}(x)}\Vert x-y\Vert_{N}\right\}. \label{useful}%
\end{equation}

It is clear that $d_{\varepsilon}>0$ which is an observation that will be very useful
in the sequel. Extract a finite subcover by such balls, which we call $\left\{  \Omega
_{i}\right\}  $. It follows that for each $\Omega_{i}$ there exists a ball of
radius $2\varepsilon$ such that the geodesic ball of radius $\varepsilon$
around any point in $\Omega_{i}$ is contained in this ball. On $\Omega_{i}$
the inequality (\ref{BNBound}) holds. Now fixing the local smooth orthonormal
tangent frame on each $\Omega_{i}$, we have a smooth map
\begin{equation}
\Omega_{i}\times B_{\varepsilon}^{d}(0)\rightarrow\Sigma
\end{equation}
given by
\begin{equation}
(x,z)\mapsto\exp_{x}(\vec{z}) \label{localmap}%
\end{equation}
where
\[
z=(z^{1},z^{2},...,z^{d})\in B_{\varepsilon}^{d}(0)
\]
is identified with the tangent vector
\[
\vec{z}=\sum_{i=1}^{d}z^{i}e_{i}(x)
\]
for the smooth orthonormal frame
\[
\{e_{1},e_{2},...,e_{d}\}.
\]

From Lemma \ref{lemGeo} \cite{BN08} we see that all computations involving
local expansions of functions will be greatly simplified by the use of normal
coordinates near a point. The fixed choice of tangent frame above allows us to
take derivatives and perform expansions in each chart (that is, the $z$ coordinates), in a way that after integrating with respect to $z$, we obtain an expression that can be differentiated with respect to $x$. This will be made clear in Section \ref{expansions}. We now set our conventions regarding the use of normal coordinates and expansions of functions near a point in terms of
their \emph{jets}. To be more precise, consider one of the local maps of the
form (\ref{localmap}). Given a function $f:\Sigma\rightarrow\mathbb{R}$
sufficiently differentiable, we let $\tilde{f}:\Omega_{i}\times B_{\varepsilon
}^{d}(0)\rightarrow\mathbb{R}$ be given by
\[
\tilde{f}(x,z)=f(\exp_{x}(z)).
\]
We will also use the following convention regarding the Taylor expansion of a
function defined on $\Sigma$ in normal coordinates near $x$. Given
$f:\Sigma\rightarrow\mathbb{R}$ in $C^{k+1}(\Sigma)$, a point $x\in\Sigma$ and
a geodesic ball $B_{\varepsilon}(x)$ as before, we have the following
expansion for $\tilde{f}$ in $B_{\varepsilon}^{d}(0)$
\begin{align*}
\tilde{f}(x,z)-\tilde{f}(x,0)=\sum_{j=1}^{k}\frac{1}{j!}D_{z}^{j}\tilde
{f}(x,0)[\underbrace{z,\ldots,z}_{j}]+\varrho_{k+1}(\tilde{f},x)[z],
\end{align*}
where each term $D_{z}^{j}\tilde{f}(x,0)[\underbrace{z,\ldots,z}_{j}]$ is a
$j$-linear form in $z$ whose coefficients are computed in terms of the $j$-th
derivatives of $f$ evaluated at $0$, and where $\varrho_{k+1}(\tilde{f},x)[z]$
satisfies a bound of the form
\begin{align}
\left\vert \varrho_{k+1}(\tilde{f},x)[z]\right\vert \leq C_{k+1}\frac{\Vert
f\Vert_{C^{k+1}(\Sigma)}}{(k+1)!}\Vert z\Vert_{d}^{k+1},
\end{align}
where $C_{k+1}$ depends on $C^{k+1}(\Sigma)$ bounds on the metric $g_{\Sigma}%
$. In fact we have
\begin{align}
\varrho_{k+1}(\tilde{f},x)[z]=\frac{(-1)^{k}}{k!}\int_{0}^{1}(1-s)^{k}%
D^{k+1}\tilde{f}(x,sz)[\underbrace{z,\ldots,z}_{k+1}]ds.
\end{align}

\begin{notation}
\label{polynote}For multilinear forms already in the abbreviated form
$D_{z}^{j}\tilde{f}(x,0)[\underbrace{z,\ldots,z}_{j}]$ we will abbreviate
further by terms of the form
\begin{align}\label{abbr}
D_{z}^{j}\tilde{f}(x,0)[z]:=\frac{1}{j!}D_{z}^{j}\tilde{f}
(x,0)[\underbrace{z,\ldots,z}_{j}].
\end{align}
\end{notation}
The choice of notation in \eqref{abbr} is to avoid the proliferation of factorials and fractions in the expansions to come. We also have the following.

\begin{definition}
Given a local map of the form (\ref{localmap}) and a $k$-differentiable
function $f$, we define the $k$-jet of $f$ at $x$ as
\[
J^{k}f(x)=\{D_{z}^{j}\tilde{f}(x,0):j=0,...,k\}.
\]

\end{definition}

Note that this depends on the local map (\ref{localmap}) and is not canonical,
as it depends on a choice of local frame. \ \ While the definition does depend
on the particular $\Omega_{i}$, we choose to suppress any dependence. \ Note
also that by $\|f\|_{C^{k}(\Sigma)}$ we mean the norm
\begin{align*}
    \|f\|_{C^{k}(\Sigma)}=\sum_{j=0}^{k}\|\nabla_{g_{\Sigma}}^{j}f\|_{C^{0}(\Sigma)} 
\end{align*}
where $\nabla^{j}_{g_{\Sigma}}$ denotes covariant derivatives with respect to the induced metric $g_{\Sigma}$. On the other hand, we remark that we may without loss of generality replace $\|f\|_{C^{k}(\Sigma)}$ by 
\begin{align}
\sum_{j\leq k}\sum_{i}\sup
_{x\in\Omega_{i}}\left\Vert D_{z}^{j}\tilde{f}(x,0)\right\Vert_{j,d}\label{tensorNorm}
\end{align}
where $\left\Vert \cdot\right\Vert_{j,d}$ in the $j\text{-}$th term in \eqref{tensorNorm} is used to denote the canonical norm on $\displaystyle{\underbrace{\bigotimes}_{j~\text{times}}\R^{d}}$.
With this set up, we expand the function $e^{\frac{\phi(z)}{2t}}.$ \ To begin,
the following is a corollary of Lemma \ref{lemGeo}.

\begin{claim}
\label{phicoeff}Let $\phi$ be the function in \eqref{define_phi} and suppose
that we are working in normal coordinates in the geodesic ball $B_{\varepsilon
}(x)$. Then we have an expansion of the form
\begin{equation}
\phi(z)=\sum_{j=4}^{k}p_{j}(\phi,x)[z]+\varrho_{k+1}(\phi,x)[z],
\label{phiexpansion}%
\end{equation}
where $p_{j}(\phi,x)[z]$ for $j=4,\ldots,k$, are homogeneous polynomials of degree $j$ and
$\varrho_{k+1}(\phi,x)[z]$ satisfies a bound of the form $\left\vert
\varrho_{k+1}(\phi,x)[z]\right\vert \leq C_{k+1}\Vert z\Vert_{d}^{k+1}$ where
$C_{k+1}$ is a constant that depends on the embedding coordinates of $\Sigma$
and $C^{k+1}(\Sigma)$ bounds on the induced metric $g_{\Sigma}$.
\end{claim}

For practical purposes, we will only expand $\phi(z)$ to fifth order, i.e. we
will use the expansion
\[
\phi(z)=p_{4}(\phi,x)[z]+p_{5}(\phi,x)[z]+\varrho_{6}(\phi,x)[z].
\]

In view of (\ref{Jordan P}) we have the identity
\[
e^{\frac{\phi(z)}{2t}}=1+\frac{\phi(z)}{2t}+\left(  \frac{\phi(z)}{2t}\right)
^{2}\int_{0}^{1}(1-s)e^{\frac{s\phi(z)}{2t}}ds.
\]
If we let
\begin{equation}
\varrho(\phi,x,t)[z]=\left(  \frac{\phi(z)}{2t}\right)  ^{2}\int_{0}%
^{1}(1-s)e^{\frac{s\phi(z)}{2t}}ds \label{Nestor}%
\end{equation}
and if again $z$ is in $B_{\varepsilon}^{d}(0)$ with $\varepsilon>0$
sufficiently small we have the bound
\begin{equation}
e^{-\frac{\Vert z\Vert_{d}^{2}}{2t}}\left\vert \varrho(\phi,x,t)[z]\right\vert
\leq C_{\Sigma}^{2}\frac{\Vert z\Vert_{d}^{8}}{8t^{2}}e^{-\frac{\Vert
z\Vert_{d}^{2}}{4t}}. \label{order8}%
\end{equation}
Here we have used 
\begin{eqnarray*}
e^{\frac{^{-\left\Vert z\right\Vert _{d}^{2}}}{2t}}e^{s\frac{^{\phi (z)}}{2t}%
} &\leq &e^{\frac{^{-\left\Vert z\right\Vert _{d}^{2}}}{2t}}e^{\frac{%
^{C_{\Sigma }\left\Vert z\right\Vert _{d}^{4}}}{2t}}\leq e^{\frac{%
^{-2\left\Vert z\right\Vert _{d}^{2}+2C_{\Sigma }\left\Vert z\right\Vert
_{d}^{4}}}{4t}} \\
&\leq &e^{\frac{^{-\left\Vert z\right\Vert _{d}^{2}}}{4t}}e^{\frac{%
^{-\left\Vert z\right\Vert _{d}^{2}+2C_{\Sigma }\left\Vert z\right\Vert
_{d}^{4}}}{4t}\leq }e^{\frac{^{-\left\Vert z\right\Vert _{d}^{2}}}{4t}}\text{
}
\end{eqnarray*}%
which holds for small $z$ where  
\[
-\left\Vert z\right\Vert _{d}^{2}+2C_{\Sigma }\left\Vert z\right\Vert
_{d}^{4}\leq 0.
\]

The bound \eqref{order8} together with \eqref{phiexpansion} suggests
considering, as before, expansions of the form
\begin{equation}
e^{\frac{\phi(z)}{2t}}=1+\frac{1}{2t}\left(  p_{4}(\phi,x)[z]+p_{5}%
(\phi,x)[z]\right)  +\frac{\varrho_{6}(\phi,x)[z]}{2t}+\varrho(\phi,x,t)[z].
\label{l98}%
\end{equation}

Similarly, the following claim is also true (cf. \cite{LP87}[Lemma 5.5]).

\begin{claim}
\label{coeff} Given the map (\ref{localmap}) let $d\mu(x,z)$ be volume form of
$g_{\Sigma}$ with respect to the exponential chart for a fixed $x$. Then for
any $k\geq2$ we have an expansion of the form
\begin{equation}
d\mu(x,z)=\left(  1+p_{2}(d\mu,x)[z]+\ldots+p_{k}(d\mu,x)[z]+\varrho
_{k+1}(d\mu,x)[z]\right)  dz \label{volE1}%
\end{equation}
where each $p_{j}(d\mu,x)[z]$ is a homogeneous polynomial in $z$ of degree $j$
(whose coefficients are smooth functions of $x$) and the remainder term
$\varrho_{k+1}(d\mu,x)[z]$ satisfies a bound of the form
\[
\left\vert \varrho_{k+1}(d\mu,x)[z]\right\vert \leq C_{k+1}\left\Vert
z\right\Vert _{d}^{k+1}%
\]
and $C_{k+1}$ is independent of $x$.
\end{claim}

The coefficients of the polynomials $p_{j}(d\mu,x)[z]$ have an interesting
interpretation in terms of curvature and its derivatives. Even though we will
not need this interpretation, it is worthwhile to mention the following

\begin{remark}
\label{remp2mu} We have the expression
\[
p_{2}(d\mu,x)[z]=-\frac{(R_{ij})_{x}z^{i}z^{j}}{3}%
\]
where $(R_{ij})_{x}$ are the components of the Ricci Curvature of $g_{\Sigma}$
at $x$. An excellent reference for these expansions is \cite[Section 5]{LP87}.
\end{remark}

\begin{remark}
 We remark one of the key conclusions from Claims \ref{phicoeff} and
\ref{coeff} which will become evident in future expansions: The intrinsic and extrinsic curvatures of the metric $g_{\Sigma}$ first appear in the expansion of the integrand
\begin{align*}
e^{-\frac{\Vert x-\exp_{x_{0}}z\Vert_{N}^{2}}{2t}}d\mu(z),
\end{align*}
as coefficients of those terms that have order $O(\left\Vert z\right\Vert _{d}^{2})$ and $O(\left\Vert z\right\Vert
_{d}^{4}/t),$\ in normal coordinates. We see below that after an appropriate rescaling these terms are essentially of the same order in $t$.
\end{remark}

\subsection{Summary of notation}

We conclude this section with a summary of the notations that we have
introduced so far (see Table \ref{table1}).

\begin{table}[h]
\centering
\begin{tabular}
[c]{|c|c|}\hline
\text{Notation} & \text{Meaning}\\\hline
& \\
$\mathbb{R}^{N}$ & Ambient space for the submanifold $\Sigma$\\
& \\
$g_{\Sigma}$ & Induced metric of the embedded submanifold $\Sigma
\hookrightarrow\mathbb{R}^{N}$\\
& \\
$\|\cdot\|_{N}$ & \text{Ambient distance}\\
& \\
$d$ & \text{Dimension of}~$\Sigma^{d}$\\
& \\
$\|\cdot\|_{d}$ & \text{Euclidean distance on $\mathbb{R}^{d}$}\\
& \\
$d\mu$ & \text{Volume form associated to the induced metric $g_{\Sigma}$ (denoted also by $d\mathrm{vol}_{\Sigma}$)}\\
       & \\
$d_{\Sigma}$ & \text{Geodesic distance on} $\Sigma$ with respect to the metric
$g_{\Sigma}$\\
& \\
$B^{d}_{\varepsilon}(0)$ & \text{Euclidean ball of radius}~$\varepsilon
$~\text{centered at}~$0$\\
& \\
$B_{\varepsilon}(x)$ & \text{Geodesic distance ball centered at}~$x$~of radius
$\varepsilon$\\
& \\
$z$ & Geodesic coordinates (in a normal coordinate neighborhood)\\
& \\
$\phi(z)$ & $d^{2}_{\Sigma}(x,y)-\|x-y\|^{2}_{N}$~for $y=\exp_{x}(z)$\\
& \\
$p_{k}(\phi,x)[z]$ & Homogeneous polynomial of degree $k$ depending\\
& on derivatives of $\phi$ at $x$ in normal coordinates\\
& \\
$p_{k}(d\mu,x)[z]$ & Homogenous polynomial of degree $k$ depending on\\
& derivatives of the volume form $d\mu$ at $x$ in normal coordinates\\
& \\
$\varrho_{k}(f,x)[z]$ & Error of order $k$ in the Taylor expansion of\\
& $f(y)-f(x)$ for $y=\exp_{x}(z)$\\
& \\
$\varrho(\phi,x,t)[z]$ & remainder term appearing in the expansion of
$e^{\frac{\phi(z)}{2t}}$ (see \eqref{Nestor})\\
& \\
$J^{k}f(x)$ & The $k$-jet of a function $f$ under a map of the form 
(\ref{localmap})\\
&\\
$\|f\|_{\mathrm{Lip}}$ & Lipschitz semi-norm defined in \eqref{lip}\\
& \\
\hline
\end{tabular}
\caption{Table of Notation}%
\label{table1}%
\end{table}

\subsection{Expansions}\label{expansions}

The follow general expansion will be extremely important. Before stating the expansion, we will adopt the following convention: the volume element is locally determined by its density function $\sqrt{\det(g_{\Sigma})}$, in particular in the $z$ coordinates we have 
\begin{align*}
    d\mu=\sqrt{\det(g_{\Sigma})}dz.
\end{align*}
We will therefore will not distinguish between $d\mu$ and its density function $\sqrt{\det(g_{\Sigma})}$ and whenever we write the jets $J^{k}d\mu$ we will actually mean the corresponding $k$-jets on the density function.

\begin{lemma}
\label{BigNasty}Let $\tilde{w}(x,z)=w(\exp_{x}z)$ on $\Omega_{i}\times
B_{\varepsilon}^{d}(0)$ for a choice of orthonormal frame and suppose $w$ $\in
C^{5}(\Sigma)$. \ Then
\begin{align}
&  \frac{1}{(2\pi t)^{d/2}}\int_{B_{\varepsilon}^{d}(0)}e^{-\frac{\Vert
x-\exp_{x}z\Vert_{N}^{2}}{2t}}\tilde{w}(x,z)d\mu(z)\nonumber\\
&  =\frac{1}{(2\pi)^{d/2}}\int_{%
\mathbb{R}
^{d}}e^{-\frac{\Vert\zeta\Vert_{d}^{2}}{2}}\left\{  \tilde{w}(x,0)+W_{1}%
(x,\zeta)t+W_{2}(x,\zeta)t^{2}\right\}  d\zeta\label{GG}\\
&  +\tilde{w}(x,0)\Phi(t^2;\phi)(x)+O(t^{5/2};\Sigma,J^{5}w,J^{5}\phi,J^{5}d\mu)\nonumber
\end{align}
where
\begin{align*}
W_{1}(x,\zeta)  &  =D_{z}^{2}\tilde{w}(x,0)[\zeta]+\tilde{w}(x,0)p_{2}%
(d\mu,x)[\zeta]+\tilde{w}(x,0)\frac{p_{4}(\phi,x)}{2}[\zeta]\\
\end{align*}
\begin{align*}
W_{2}(x,\zeta)  &  =\left\{
\begin{array}
[c]{c}%
\begin{array}
[c]{c}%
D_{z}^{4}\tilde{w}(x,0)[\zeta]+D_{z}^{2}\tilde{w}(x,0)\otimes p_{2}%
(d\mu,x)[\zeta]\\
+D_{z}^{1}\tilde{w}(x,0)\otimes p_{3}(d\mu,x)[\zeta]+\tilde{w}(x,0)p_{4}%
(d\mu,x)[\zeta]\\%
\begin{array}
[c]{c}%
+\tilde{w}(x,0)p_{4}(\phi,x)\otimes p_{2}(d\mu,x)[\zeta]\\
+p_{4}(\phi,x)\otimes D_{z}^{2}\tilde{w}(x,0)[\zeta]\\
+p_{5}(\phi,x)\otimes D_{z}\tilde{w}(x,0)[\zeta]
\end{array}
\end{array}
\end{array}
\right\}
\end{align*}
and%
\begin{align}
\left\vert \Phi(t^2;\phi)(x)\right\vert \leq Ct^{2}\label{defPhi}
\end{align}
for a constant $C$ depending only on the geometry of $\Sigma$ and $\phi, $and

\[
\left\vert O(t^{5/2};\Sigma,J^{5}w,J^{5}\phi,J^{5}d\mu)\right\vert \leq Ct^{5/2}%
\]

is a function that satisfies a bound 
for a (different)$\ C$ that depends on $\Sigma$ and the 5-jets of the function
$w,\phi$ and the volume density $d\mu$ in terms of normal coordinate expansions.
\end{lemma}

We will prove Lemma \ref{BigNasty} in the appendix. Note that the statement of Lemma \ref{BigNasty} involves the term $O(t^{5/2};J^{5}w,J^{5}\phi,J^{5}d\mu)$. In general we will use the notation 
\begin{align}\label{generalO}
    O(\psi;\Sigma,J^{k_1}f_1,\ldots,J^{k_n}f_{n}),
\end{align}
to denote a function $u$ that satisfies at every point a bound of the form
\begin{align*}
    \left|u\right|\le C\cdot \left|\psi\right|,
\end{align*}
where $C$ depends on depends on the embedding coordinates, the geometry of $\Sigma$, and $C^{0}(\Sigma)$ bounds on jets of functions $J^{k_1}f_{1},\ldots,J^{k_n}f_{n}$. Note that both $u$ and $\psi$ could be defined globally or locally and it could depend on variables apart from $t$ and $x$. 
\begin{remark}
\label{Thomas}The integrals of the polynomials in (\ref{GG}) define
smooth and globally defined functions of $x.$ While the normal coordinate
chart involves a choice up to the action of the orthogonal group $\mathcal{O}(d,\R)$ (and therefore the  polynomials\ $W_{i}$ are not uniquely defined ), the integral in this case is invariant under these actions and is therefore well defined globally. \ 
\end{remark}

We immediately use Lemma \ref{BigNasty} to expand $\theta_{t}$,  recalling (\ref{generaldensity}). 

\begin{lemma}
\label{theta_expand}%
\[
\frac{\theta_{t}(x)}{(2\pi t)^{d/2}}=1+\Upsilon_{1}(x)t+\Upsilon_{2}%
(x)t^{2}+\Upsilon_{2+}(x,t).
\]
where
\begin{align}
\Upsilon_{1}(x) &  =\frac{1}{(2\pi)^{d/2}}\int_{%
\mathbb{R}
^{d}}e^{-\frac{\Vert\zeta\Vert_{d}^{2}}{2}}\left\{  p_{2}(d\mu,x)[\zeta
]+\frac{p_{4}(\phi,x)}{2}[\zeta]\right\}  d\zeta\nonumber\\
\Upsilon_{2}(x) &  =\frac{1}{(2\pi)^{d/2}}\int_{%
\mathbb{R}
^{d}}e^{-\frac{\Vert\zeta\Vert_{d}^{2}}{2}}\left\{
\begin{array}
[c]{c}%
p_{4}(d\mu,x)[\zeta]\\
+p_{4}(\phi,x)\otimes p_{2}(d\mu,x)[\zeta]
\end{array}
\right\}  d\zeta\nonumber\\
\left\vert \Upsilon_{2+}(x,t)\right\vert  &  \leq C_{\Sigma}t^{2}.
\label{Bon Scott}%
\end{align}
In particular, $\Upsilon_{1}(x)$ and $\Upsilon_{2}(x)$ are smooth and globally
defined. (See Remark \ref{Thomas}).
\end{lemma}

\begin{notation}
Since $\phi$ and $d\mu$ are assumed to be smooth, we will abbreviate
$C_{\Sigma}$ to denote a constant depending on $\Sigma$ and perhaps some higher
jets on $\phi$ and $d\mu$, as in (\ref{Bon Scott}). 
\end{notation}

\begin{proof}
Apply Lemma \ref{BigNasty},  with $w\equiv 1$ and therefore $\tilde{w}\equiv 1.$%
\begin{align*}
&  \frac{1}{(2\pi t)^{d/2}}\int_{\Sigma}e^{-\frac{\Vert x-y\Vert_{N}^{2}}{2t}%
}d\mu(y)\\
&  =\frac{1}{(2\pi t)^{d/2}}\int_{B_{\varepsilon}(x)}e^{-\frac{\Vert
x-\exp_{x}z\Vert_{N}^{2}}{2t}}d\mu(z)+\frac{1}{(2\pi t)^{d/2}}\int%
_{\Sigma\backslash B_{\varepsilon}(x)}e^{-\frac{\Vert x-y\Vert_{N}^{2}%
}{2t}}d\mu(y)\\
&  =\frac{1}{(2\pi)^{d/2}}\int_{%
\mathbb{R}
^{d}}e^{-\frac{\Vert\zeta\Vert_{d}^{2}}{2}}\left\{
\begin{array}
[c]{c}%
1\\
+\left(  p_{2}(d\mu,x)[\zeta]+\tilde{w}(x,0)\frac{p_{4}(\phi,x)}{2}%
[\zeta]\right)  t\\
+\left(  \tilde{w}(x,0)p_{4}(d\mu,x)[\zeta]+p_{4}(\phi,x)\tilde{w}%
(x,0)p_{2}(d\mu,x)[\zeta]\right)  t^{2}%
\end{array}
\right\}  d\zeta\\
&+\Phi(t^{2};\phi)+O(t^{5/2};\Sigma,\phi,J^{5}d\mu)+O(e^{-\varepsilon
^{2}/2t};\Sigma),
\end{align*}
Collecting and absorbing the terms completes the proof.

\end{proof}
Recall that we are
assuming the manifold is smooth, thus the distance function and\ $\phi$ are
smooth within the injectivity radius of the diagonal and the expression for
$d\mu$ is also smooth. 
\begin{corollary}
\label{Deep Purple} As above,
\[
\frac{(2\pi t)^{d/2}}{\theta_{t}(x)}=1-\Upsilon_{1}(x)t+\hat{\Upsilon}%
_{2}(t,x),
\]
where
\[
\left\vert \hat{\Upsilon}_{2}(t,x)\right\vert \leq C_{\Sigma}t^{2}.
\]

\end{corollary}

\begin{proof}
Let
\[
\mathscr{J}(x,t)=\frac{(2\pi t)^{d/2}}{\theta_{t}(x)}-\left(  1-\Upsilon_{1}(x)t\right)
\]
so that
\[
\frac{(2\pi t)^{d/2}}{\theta_{t}(x)}=1-\Upsilon_{1}(x)t+\mathscr{J}(x,t).
\]

Then
\begin{align*}
\frac{\theta_{t}(x)}{(2\pi t)^{d/2}}\times\frac{(2\pi t)^{d/2}}{\theta_{t}(x)}
&  =1\\
&  =\left(  1+\Upsilon_{1}(x)t+\Upsilon_{2}(t,x)+\Upsilon_{2+}(x,t)\right)
\left(  1-\Upsilon_{1}(x)t+\mathscr{J}(x,t)\right)  \\
&  =1+\mathscr{J}(x,t)\left(1+O(t;\Sigma)\right)+O(t^{2};\Sigma)
\end{align*}
which implies
\[
\mathscr{J}(x,t)=O(t^{2};\Sigma).
\]

\end{proof}

\bigskip

Now we can compute an expression for the Laplace operator.

\begin{lemma}
\label{Laplacef} \ Suppose $f\in C^{k+4}\left(  \Sigma\right)  $ for $k\geq1.$
Fix any $x_{0}\in\Sigma.$ Then
\[
L_{t}f=\Delta_{g}f(x)+F_{1}(x;f)t-\Upsilon_{1}(x)t+\hat{F}_{3/2}(x,t;f)
\]
where $F_{1}$ is a $C^{k}(\Sigma)$ locally determined function depending on
the geometry of $\ \Sigma$ and $J^{4}f;$ $\Upsilon_{1}$ is as in Corollary \ref{Deep Purple}; and
\begin{equation}
\left\vert \hat{F}_{3/2}(x,t;f)\right\vert \leq C(\Sigma,\|f\|_{C^{5}(\Sigma)})t^{3/2}
\label{F_3}%
\end{equation}

\end{lemma}

\begin{proof}
Recall%
\[
L_{t}f(x)=\frac{2}{t}\frac{(2\pi t)^{d/2}}{\theta_{t}(x)}\frac{1}{(2\pi
t)^{d/2}}\int_{\Sigma}e^{-\frac{\Vert x-y\Vert_{N}^{2}}{2}}\left(
f(y)-f(x)\right)  d\mu(y).
\]
Splitting the integral into an integral over a geodesic ball and an integral over the rest of
$\Sigma$
\begin{align*}
\frac{1}{(2\pi t)^{d/2}}\int_{\Sigma}&e^{-\frac{\Vert x-y\Vert_{N}^{2}}{2t}%
}\left(  f(y)-f(x)\right)  d\mu(y)=\\
&\frac{1}{(2\pi t)^{d/2}}\int%
_{B_{\varepsilon}^{d}(0)}e^{-\frac{\Vert x-\exp_{x}z\Vert_{N}^{2}}{2t}}\left(
\tilde{f}(x,z)-\tilde{f}(x,0)\right)  d\mu(z)\\
&  +\frac{1}{(2\pi t)^{d/2}}\int_{\Sigma\backslash B_{\varepsilon}^{d}%
(0)}e^{-\frac{\Vert x-y\Vert_{N}^{2}}{2t}}\left(  f(y)-f(x)\right)  d\mu(y)
\end{align*}
we can use Lemma \ref{BigNasty} by setting $\tilde{w}(x,z)=\tilde{f}(x,z)-\tilde{f}(x,0)$, and noting that $\tilde{w}(x,0)=0$ we have
\begin{align}
&  \frac{1}{(2\pi t)^{d/2}}\int_{B_{\varepsilon}^{d}(0)}e^{-\frac{\Vert
x-\exp_{x}z\Vert_{N}^{2}}{2t}}\left(  \tilde{f}(x,z)-\tilde{f}(x,0)\right)
d\mu(z)\nonumber\\
&  =\frac{1}{(2\pi)^{d/2}}\int_{%
\mathbb{R}
^{d}}e^{-\frac{\Vert\zeta\Vert_{d}^{2}}{2}}\left\{
\begin{array}
[c]{c}%
D_{z}^{2}\tilde{f}(x,0)[\zeta]t\\
+\left(
\begin{array}
[c]{c}%
D_{z}^{4}\tilde{f}(x,0)[\zeta]+D_{z}^{2}\tilde{f}(x,0)\otimes p_{2}%
(d\mu,x)[\zeta]\\
+D_{z}^{1}\tilde{f}(x,0)\otimes p_{3}(d\mu,x)[\zeta]
\end{array}
\right)  t^{2}%
\end{array}
\right\}  d\zeta\label{Purple Haze}\\
+ &  O(t^{5/2};J^{5}f).\nonumber
\end{align}
\ We collect the second order terms in $t$ into $F_{1}$ and absorb the higher order into
$F_{3/2}$ to obtain the following expression.
\begin{align*}
\frac{1}{(2\pi t)^{d/2}}\int_{B_{\varepsilon}^{d}(0)}e^{-\frac{\Vert
x-y\Vert_{N}^{2}}{2t}}\left(  f(x)-f(y)\right)  d\mu(y) &  =\frac{t}%
{(2\pi)^{d/2}}\int_{%
\mathbb{R}
^{d}}e^{-\frac{\Vert\zeta\Vert_{d}^{2}}{2}}D_{z}^{2}\tilde{f}(x,0)[\zeta
]d\zeta \\
&  +F_{1}(x)t^{2}+F_{3/2}(x,t).
\end{align*}
For $\left\vert F_{3/2}(x,t)\right\vert \leq Ct^{5/2}.$ \ Recalling notation
\ref{polynote} and using Lemma \ref{basic_monomials}
\begin{equation}
\frac{1}{(2\pi)^{d/2}}\int_{%
\mathbb{R}
^{d}}e^{-\frac{\Vert\zeta\Vert_{d}^{2}}{2}}D_{z}^{2}f(x,0)[\zeta]d\zeta=\frac{1}%
{2}\Delta_{g}f(x)\label{la verdad}%
\end{equation}

Now using Corollary \ref{Deep Purple}, we combine%
\begin{align}
&  \frac{2}{t}\frac{(2\pi t)^{d/2}}{\theta_{t}(x)}\frac{1}{(2\pi t)^{d/2}}%
\int_{\Sigma}e^{-\frac{\Vert x-y\Vert_{N}^{2}}{2}}\left(  f(y)-f(x)\right)
d\mu(y)\label{willrepeat}\\
&  =\frac{2}{t}\left(  1-\Upsilon_{1}(x)t+\hat{\Upsilon}_{2}(t,x)\right)
\left\{
\begin{array}
[c]{c}%
\frac{1}{2}\Delta_{g}f(x)t+F_{1}(x)t^{2}+F_{3/2}(x,t)\\
+\frac{1}{(2\pi t)^{d/2}}\int_{\Sigma\backslash B_{\varepsilon}%
(x)}e^{-\frac{\Vert x-y\Vert_{N}^{2}}{2t}}\left(  f(x)-f(y)\right)  d\mu(y)
\end{array}
\right\}  \nonumber\\
&  =\Delta_{g}f(x)+2F_{1}(x)t-2\Upsilon_{1}(x)t+\hat{F}_{3/2}(x,t)\nonumber
\end{align}
where we have consolidated all terms of order higher than $3/2.$
\end{proof}

In what follows we will make use of the following notation
\begin{align}
    \|f\|_{\mathrm{Lip}}=\sup_{x,y\in \Sigma, ~ x\ne y}\left\{\frac{|f(x)-f(y)|}{\|x-y\|_{N}}\right\}\label{lip}.
\end{align}
Note that \eqref{lip} represents a a Lipschitz semi-norm with respect to the ambient metric. We now follow suit for $\Gamma:$

\begin{lemma}
\label{cdc1}Assume that $f,h$ are $C^{k+2}$ for $k\leq1.$%
\begin{equation}
\Gamma_{t}(f,h)=\nabla f\cdot\nabla h+H_{1}(x;f,h)t-\Upsilon_{1}(x)t+\hat
{H}_{3/2}(x,t;f,h)\label{Pink Floyd}%
\end{equation}
where $H_{1}\in C^{k-1}(\Sigma)$, in fact $H_{1}$ is a quadratic expression in the jets $J^{3}f(x)$ and $J^{3}h(x)$ at the point $x$; $\Upsilon_{1}$ is as in Lemma \ref{theta_expand} ; and
\[
\left\vert \hat{H}_{3/2}(x,t;f,h)\right\vert \leq C(\Sigma,\|%
f\|_{C^{4}(\Sigma)},\|h\|_{C^{4}(\Sigma)})\cdot t^{3/2}.
\]
Alternatively, assuming that $f,h$ are merely $C^{1}$, we have
\begin{equation}
\left\vert \Gamma_{t}(f,h)\right\vert (x)\leq2\left(  1-\Upsilon_{1}%
(x)t+\hat{\Upsilon}_{2}(t,x)\right)  C\left\Vert f\right\Vert _{\mathrm{Lip}}\left\Vert
h\right\Vert _{\mathrm{Lip}}\label{Cherry Thunder Fuck}%
\end{equation}
for some $C$ depending only on the global geometry of $\Sigma.$
\end{lemma}

\begin{proof}
Recall%
\[
\Gamma_{t}(f,h)(x)=\frac{1}{2}\frac{1}{t}\frac{(2\pi t)^{d/2}}{\theta_{t}%
(x)}\frac{1}{(2\pi t)^{d/2}}\int_{\Sigma}e^{-\frac{\Vert x-y\Vert_{N}^{2}}%
{2t}}\left(  f(x)-f(y)\right)  \left(  h(x)-h(y)\right)  d\mu(y).
\]

Letting
\[
w(x,y)=\left(  f(x)-f(y)\right)  \left(  h(x)-h(y)\right)
\]
we obtain%
\begin{align}
\frac{1}{(2\pi t)^{d/2}}\int_{\Sigma}e^{-\frac{\Vert x-y\Vert_{N}^{2}}{2t}%
}w(x,y)d\mu(y) &  =\frac{1}{(2\pi t)^{d/2}}\int_{B_{\varepsilon}%
^{d}(0)}e^{-\frac{\Vert x-\exp_{x}z\Vert_{N}^{2}}{2t}}\tilde{w}(x,z)d\mu
(z)\label{Ghost Bubba}\\
&  +\frac{1}{(2\pi t)^{d/2}}\int_{\Sigma\backslash B_{\varepsilon}^{d}%
(0)}e^{-\frac{\Vert x-y\Vert_{N}^{2}}{2t}}w(x,y)d\mu(y)\nonumber
\end{align}
where $\tilde{w}(x,z)=w(x,\exp_{x}(z))$ for $y$ in $B_{\epsilon}(x)$. We can now apply Lemma \ref{BigNasty}, noting that $\tilde{w}(x,0)=0$ and
$D_{z}\tilde{w}(x,0)=0$ \ (thus the $4$-jet of $w$ depends only on the $2$-jet
of $f$ or $h.$)$\ $
\begin{align*}
&  \frac{1}{(2\pi t)^{d/2}}\int_{B_{\varepsilon}^{d}(0)}e^{-\frac{\Vert
x-\exp_{x}z\Vert_{N}^{2}}{2t}}\tilde{w}(x,0)d\mu(z)\\
&  =\frac{1}{(2\pi)^{d/2}}\int_{%
\mathbb{R}
^{d}}e^{-\frac{\Vert\zeta\Vert_{d}^{2}}{2}}\left\{
\begin{array}
[c]{c}%
D_{z}^{2}\tilde{w}(x,0)[\zeta]t\\
+\left(  D_{z}^{4}\tilde{w}(x,0)[\zeta]+D_{z}^{2}\tilde{w}(x,0)\otimes
p_{2}(d\mu,x)[\zeta]\right)  t^{2}%
\end{array}
\right\}  d\zeta\\
&  +O(t^{5/2};\Sigma,J^5w,J^5\phi,J^{5}d\mu).
\end{align*}
Noting that
\[
D^{2}_{z}\tilde{w}(x,0)=2D_{z}\tilde{f}(x,0)\otimes D_{z}\tilde{h}(x,0)
\]
and absorbing the $O(t^{5/2})$ terms as we have been doing, we obtain
\[
\frac{1}{(2\pi t)^{d/2}}\int_{B_{\varepsilon}^{d}(0)}e^{-\frac{\Vert
x-y\Vert_{N}^{2}}{2t}}\tilde{w}(x,y)d\mu(y)=\left(  \nabla f\cdot\nabla
h\right)  (x)t+H_{1}(x)t^{2}+H_{3/2}(x,t),
\]
with $\left\vert H_{3/2}(x,t)\right\vert \leq Ct^{5/2}.$ Observe that $H_{1}(x)$ is in fact a quadratic expression in the components of the jets $J^{3}f(x)$ and $J^{3}h(x)$ at the point $x$. To see this, note that 
\begin{align*}
    H_{1}(x;f,h)=\frac{1}{(2\pi)^{d/2}}\int_{\R^{d}}e^{-\frac{\Vert\zeta\Vert_{d}^{2}}{2}}\left(  D_{z}^{4}\tilde{w}(x,0)[\zeta]+D_{z}^{2}\tilde{w}(x,0)\otimes
p_{2}(d\mu,x)[\zeta]\right)d\zeta
\end{align*}
which implies that in principle, $H_1$ is a quadratic expression in $J^{4}f(x)$ and $J^{4}h(x)$, however upon examining the term $D_{z}^{4}\tilde{w}(x,0)[\zeta]$ one sees that 
\begin{align*}
 D_{z}^{4}\tilde{w}(x,0)[\zeta]=\sum_{\alpha+\beta=4}\left(D_{z}^{\alpha}(\tilde{f}(x,z)-\tilde{f}(x,0))[\zeta]\otimes D^{\beta}_{z}(\tilde{h}(x,z)-\tilde{h}(x,0))[\zeta]\right)\bigg\rvert_{z=0}   
\end{align*}
which implies that the only two terms that contain derivatives of order four of $\tilde{f}(x,z)$ and $\tilde{h}(x,z)$ are 
\begin{align*}
D^{4}_z(\tilde{f}(x,z)-\tilde{f}(x,0))[z]\otimes(\tilde{h}(x,z)-\tilde{h}(x,0))[z]\bigg\rvert_{z=0} = 0,
\end{align*}
and
\begin{align*}
(\tilde{f}(x,z)-\tilde{f}(x,0))[z]\otimes D^{4}_{z}(\tilde{h}(x,z)-\tilde{h}(x,0))[z]\bigg\rvert_{z=0} = 0.
\end{align*}
It follows that we may represent $H_{1}$ schematically as 
\begin{align}\label{quadraticH1}
    H_{1}(x;f,h) = J^{3}f(x)*J^{3}h(x),
\end{align}
i.e., an expression that is quadratic in the derivatives of order at most 3 of $f(x)$ and $h(x)$. Estimate (\ref{Pink Floyd}) follows now as for Laplacian, see
(\ref{willrepeat}).  Alternatively, we may go back to (\ref{Ghost Bubba}) and instead apply a
Lipschitz bound
\[
\left\vert \tilde{w}(x,y)\right\vert =\left\vert \left(  f(x)-f(y)\right)
\left(  h(x)-h(y)\right)  \right\vert \leq C\left\Vert f\right\Vert
_{\mathrm{Lip}}\left\Vert h\right\Vert _{\mathrm{Lip}}\left\Vert x-y\right\Vert_{N} ^{2}.
\]
We therefore have the estimate
\begin{align}
\left\vert \frac{1}{(2\pi t)^{d/2}}\int_{\Sigma}e^{-\frac{\Vert x-y\Vert
_{N}^{2}}{2t}}\tilde{w}(x,y)d\mu(y)\right \vert &\leq\frac{C}{(2\pi
t)^{d/2}}\int_{\Sigma}e^{-\frac{\Vert x-y\Vert_{N}^{2}}{2t}}\left\Vert
f\right\Vert _{\mathrm{Lip}}\left\Vert h\right\Vert _{\mathrm{Lip}}\left\Vert x-y\right\Vert_{N} ^{2}d\mu(y)\label{Wonder Kid}
\end{align}
Note that for $y$ in $B_{\epsilon}(x)$ we may write $y=\exp_{x}(z)$ and use \eqref{define_phi} to conclude that
\begin{align}
 \frac{1}{(2\pi
t)^{d/2}}\int_{\Sigma}e^{-\frac{\Vert x-y\Vert_{N}^{2}}{2t}}\left\Vert x-y\right\Vert_{N} ^{2}d\mu(y)=t+O(t^2)\label{intDistsquare}   
\end{align}
and combining \eqref{Wonder Kid} with \eqref{intDistsquare} we have 
\begin{align}
\left\vert \frac{1}{(2\pi t)^{d/2}}\int_{\Sigma}e^{-\frac{\Vert x-y\Vert
_{N}^{2}}{2t}}\tilde{w}(x,y)d\mu(y)\right \vert\le 2t\|f\|_{\mathrm{Lip}}\|h\|_{\mathrm{Lip}}.    
\end{align}

Multiplying by
\begin{equation}
\frac{1}{2}\frac{1}{t}\frac{(2\pi t)^{d/2}}{\theta_{t}(x)}=\frac{2}{t}\left(
1-\Upsilon_{1}(x)t+\hat{\Upsilon}_{2}(t,x)\right)  \label{Wunder Kind}%
\end{equation}
gives
\begin{align*}
\left\vert \frac{1}{(2\pi t)^{d/2}}\int_{\Sigma}e^{-\frac{\Vert x-y\Vert
_{N}^{2}}{2t}}\tilde{w}(x,y)d\mu(y)\right\vert \leq 4C\left(  1-\Upsilon
_{1}(x)t+\hat{\Upsilon}_{2}(t,x)\right)  \left\Vert f\right\Vert
_{\mathrm{Lip}}\left\Vert h\right\Vert _{\mathrm{Lip}}.
\end{align*}

\end{proof}

\bigskip We may now iterate, but first, there is a small technical difficulty that we need to address: \ We have expressed our original expansion under the assumption that
$w\in C^{5}(\Sigma)$, and thus formulas for $L_{t}$ and $\Gamma_{t}$ to apply whenever
$f$ is five times differentiable,  and in fact the remainder terms are bounded explicitly 
by expressions in $J^{5}f.$ Applying a differential operator to a finitely differentiable
function lowers the differentiability order of the function, so these bounds are 
no longer available when we try
to iterate. One option is to make an effort to carefully restate the
dependence on differentiability (essentially by either repeating the proof of the
expansion Lemma \ref{BigNasty}, or, by proving a very general version of Lemma
\ref{BigNasty} \ which we only use in two cases). Instead, our approach
is to use a local approximation, as is suggested by the following Lemma.  

\begin{lemma}
\label{smooth_local}Suppose that $f$ is $C^{3}$ function on $\Sigma.$ Then for
a fixed $x_{0},$ we may write
\[
f(x)=\alpha(x)+\beta(x)
\]
where $\alpha$ is a smooth function osculating $f$ to second order at $x_{0}$
that is
\begin{align}
J^{2}\alpha(x_{0}) &  =J^{2}f(x_{0})\label{jetsagree}\\
\left\vert \beta(x)\right\vert  &  \leq C(\Sigma,\left\Vert f\right\Vert _{C^{3}(\Sigma)
})\left\Vert x-x_{0}\right\Vert ^{3}_{N}.\nonumber
\end{align}
where \eqref{jetsagree} holds in the normal coordinate chart defined in $B_{\epsilon}(x_{0})$. The constant $C(\Sigma,\left\Vert f\right\Vert _{C^{3}(\Sigma)})$ can be chosen
independent of $x_{0}.$
\end{lemma}

\begin{proof}
Choose normal coordinates at $x_{0}$ of radius $\varepsilon.$ \ \ Let
$\eta(x)$ be a smooth bump function which is unity on a ball of radius
$\varepsilon/2$ and vanishing inside the ball. \ \ Let $f_{2}$ be the second
order Taylor polynomial at the origin. \ Then on $B_{\varepsilon}(0)$ we have
\[
\left\vert f(z)-f_{2}(z)\right\vert \leq C_{3}\cdot\left\Vert z\right\Vert _{d}^{3}.
\]
Now define
\begin{align*}
\alpha(z) &  =\eta(z)f_{2}(z)\\
\beta(x) &  =\left(  f(x)-f_{2}(x)\right)  \eta(x)+(1-\eta)f(x).
\end{align*}
Clearly inside $B_{\varepsilon/2}$ we have
\begin{equation}
\left\vert \beta(x)\right\vert =\left\vert \left(  f(x)-f_{2}(x)\right)
\eta(x)\right\vert \leq C_{3}\cdot\left\Vert z\right\Vert _{d}^{3}\leq C(\left\Vert
f\right\Vert _{C^{3}(\Sigma)},\Sigma)\left\Vert x-x_{0}\right\Vert ^{3}_{N}%
\label{beta_bound}%
\end{equation}
for a $C$ that depends on $\Sigma$, including a small metric distortion
between ambient and geodesic distance occurring on small balls (see the discussion surrounding \eqref{define_phi}).  The inequality holds on the outside of $B_{\varepsilon
/2}$ as $\left\Vert x-x_{0}\right\Vert ^{3}_{N}$ is bounded below on
$\Sigma\backslash B_{\varepsilon/2}$ by virtue of line (\ref{useful}). 
\end{proof}

\begin{lemma}
\label{l319} Suppose $f$ $\in$ $C^{5}(\Sigma).$ \ Then%
\[
L_{t}(\Gamma_{t}(f,f))(x)=\frac{1}{2}\Delta_{g}\left\vert \nabla f\right\vert
^{2}(x)+A_{1/2}(x,t;f)
\]
for 
\[
\left|A_{1/2}(x,t;f)\right|\leq C(\Sigma,\|f\|_{C^{5}(\Sigma)})t^{1/2}.
\]

\end{lemma}

\begin{proof}
Recall that for $f\in C^{5}(\Sigma)$ we have by Lemma \ref{cdc1}%
\[
\Gamma_{t}(f,f)(x)=(\nabla f\cdot\nabla f)(x)+H_{1}(x;f)t-\Upsilon_{1}%
(x)t+\hat{H}_{3/2}(x,t;f)
\]
with
\[
\left\vert \hat{H}_{3/2}(x,t;f)\right\vert \leq C(\Sigma,\|f\|_{C^{4}(\Sigma)})\cdot t^{3/2}%
\]
where $\Upsilon_{1}(x)$ is smooth, $\left\vert \nabla f\right\vert ^{2}(x)$ is
$C^{4}(\Sigma)$, and $H_{1}(x)$ is $C^{3}(\Sigma).$ \ In particular at any given point
we let%
\begin{align*}
S(x) &  :=\nabla f\cdot\nabla f(x)+H_{1}(x;f)t-\Upsilon_{1}(x)t\\
&  =\alpha(x)+\beta(x)-\Upsilon_1(x)\cdot t+H_{1}(x;f)\cdot t,
\end{align*}
where $\alpha(x)$ is $C^{5}(\Sigma)$ (smooth, if we want) and 
\begin{align*}
\left\vert \beta(x)\right\vert \leq C\left\Vert \langle\nabla f,\nabla f\rangle_{g}(x)\right\Vert _{C^{3}(\Sigma)}%
\left\Vert x-x_{0}\right\Vert ^{3}_{N}.
\end{align*}

Let us start by computing
\begin{align*}
L_{t}\Gamma_{t}(f,f)(x)&=L_{t}\left(  \alpha_{0}(x)+\beta_{0}(x)+\hat{H}_{3/2}
(x,t;f)\right)\\
&+L_{t}\left(-\Upsilon_1(x)\cdot t+H_{1}(x;f)\cdot t\right)
\end{align*}
as follows. Applying Lemma \ref{Laplacef} and (\ref{jetsagree}) we have
\begin{align*}
L_{t}\left(  \alpha\right)  (x_{0}) &  =\Delta_{g}\alpha(x_{0})+F_{1}%
(x_{0};\alpha)t-\Upsilon_{1}(x_{0})t+\hat{F}_{3/2}(x_{0},t;\alpha)\\
&  =\Delta_{g}\left\vert \nabla f\right\vert ^{2}(x_{0})+\Delta_{g}\left[
H_{1}(x_{0};f)-\Upsilon_{1}(x_{0})\right]  t\\
&  +F_{1}(x_{0};\alpha)t-\Upsilon_{1}(x_{0})t+\hat{F}_{3/2}(x_{0},t;\alpha).
\end{align*}
Next using (\ref{beta_bound}) and rescaling 
\begin{align}
\left\vert L_{t}\beta(x_{0})\right\vert  &  =\left\vert \frac{1}{t}\frac{(2\pi
t)^{d/2}}{\theta_{t}(x)}\frac{1}{(2\pi t)^{d/2}}\int_{\Sigma}e^{-\frac{\Vert
y-x_{0}\Vert_{N}^{2}}{2t}}\left(  \beta(y)-\beta(x_{0})\right)  d\mu
(y)\right\vert\label{betaline1} \\
&  \leq\frac{6\sqrt{6}2^{\frac{d+4}{4}}\cdot e^{-3/2}}{t}\left\vert \frac{(2\pi t)^{d/2}}{\theta_{t}(x)}\frac
{1}{(2\pi)^{d/2}}\int_{\Sigma}C_{3}e^{-\frac{\Vert y-x_{0}\Vert_{N}^{2}}{2}}%
t^{3/2}d\mu(y)\right\vert\label{betaline2} \\
&  \leq C_3t^{1/2}\left|1-\Upsilon_{1}(x)t+\Upsilon_{2}(x,t)\right|.\label{betaline3}
\end{align}
In order to bound line \eqref{betaline1} by line \eqref{betaline2} we have used the inequality
\begin{align*}
 \left|e^{-\frac{\Vert
y-x_{0}\Vert_{N}^{2}}{2t}}\left(  \beta(y)-\beta(x_{0})\right) \right|&\le C_{3}e^{-\frac{\Vert
y-x_{0}\Vert_{N}^{2}}{2t}}\|y-x_{0}\|^{3}_{N}\\
&=C_{3}e^{-\frac{\Vert y-x_{0}\Vert_{N}^{2}}{4t}}e^{-\frac{\Vert y-x_{0}\Vert_{N}^{2}}{4t}}\|y-x_{0}\|^{3}_{N}\\
&\le C_36\sqrt{6}e^{-3/2}t^{3/2}e^{-\frac{\Vert y-x_{0}\Vert_{N}^{2}}{4t}},
\end{align*}
and for $t$ sufficiently small depending only on $\varepsilon$ we have
\begin{align*}
    \frac{\displaystyle{\int_{\Sigma}e^{-\frac{\|y-x_{0}\|^2}{4t}}d\mu(y)}}{\theta_{t}(x)}\le 2^{\frac{d+2}{2}}.
\end{align*}
Note that we have absorbed all multiplicative constants in line \eqref{betaline2} into $C_{3}$ in \eqref{betaline3}. We are now able to compute $L_{t}\hat{H}_{3/2}(x,t)$. We start by estimating $L_{t}H_{1}(x)$ by first writing
\begin{align}
   \frac{1}{(2\pi t)^{d/2}}\int_{\Sigma}e^{-\frac{\|x-y\|^2_{N}}{2t}}(H_{1}(y)&-H_{1}(x))d\mu(y)\nonumber\\
   &=\frac{1}{(2\pi t)^{d/2}}\int_{B_{\epsilon}(x)}e^{-\frac{\|x-y\|^2_{N}}{2t}}(H_{1}(y)-H_{1}(x))d\mu(y)\nonumber\\
   &+\frac{1}{(2\pi t)^{d/2}}\int_{\Sigma\setminus B_{\epsilon}(x)}e^{-\frac{\|x-y\|^2_{N}}{2t}}(H_{1}(y)-H_{1}(x))d\mu(y)\label{intH1}
\end{align}
and for the integral over $B_{\epsilon}(x)$ in \eqref{intH1} we estimate $H_{1}(y)-H_{1}(x)$ by first writing $y=\exp_{x}(z)$ and also
\begin{align}
H_{1}(y)-H_{1}(x)&=\int_{0}^{1}\frac{d}{dt}H_{1}(\exp_{x}(tz))dt\nonumber\\
&=\int_{0}^{1}D^{1}_{z}\tilde{H}_{1}(x,tz)\frac{d}{dt}\exp_{x}(tz)dt.\label{H1diff}
\end{align}
We therefore have
\begin{align}
    \left|H_{1}(y)-H_{1}(x)\right|&\le \sqrt{\left(\int_{0}^1\left\|D_{z}\tilde{H}_{1}(x,tz)\right\|^{2}_{\dot{\gamma}(t)}dt\right)}d_{\Sigma}(x,y)\label{geodesic_estimate}\\
    &=C(\Sigma,\|f\|_{C^{4}(\Sigma)},\|h\|_{C^{4}(\Sigma)})\cdot\|z\|_{d},\nonumber
\end{align}
where $\dot{\gamma}(y)=\frac{d}{dt}\exp_{x}(tz)$. Note that the inequality
\begin{align*}
\sqrt{\left(\int_{0}^1\left\|D_{z}\tilde{H}_{1}(x,tz)\right\|^{2}_{\dot{\gamma}(t)}dt\right)}\le C(\Sigma,\|f\|_{C^{4}(\Sigma)},\|h\|_{C^{4}(\Sigma)})    
\end{align*}
follows from the observation that $H_{1}(x)$ is given by a quadratic polynomial on the components of the jets $J^{3}f(x)$ and $J^{3}h(x)$ at the point $x$ (see the discussion in the proof of Lemma \ref{cdc1} and in particular equation \eqref{quadraticH1}). By Lemma \ref{BigNasty} we have 
\begin{align*}
\left| \frac{1}{(2\pi t)^{d/2}}\int_{\Sigma}e^{-\frac{\|x-y\|^2_{N}}{2t}}(H_{1}(y)-H_{1}(x))d\mu(y)\right|= O(t^{1/2};J^{4}f(x),J^{4}h(x)).
\end{align*}
Finally, in order to estimate $\Gamma_{t}\left(-\Upsilon_{1}(x)\cdot t\right)$ we write
\begin{align*}
L_{t}\left(-\Upsilon_{1}(x)\cdot t\right)=\frac{1}{\theta_{t}(x)}\int_{\Sigma}e^{-\frac{\|x-y\|^2_{N}}{2t}}(\Upsilon_1(x)-\Upsilon_1(y))d\mu(y)    
\end{align*}
and an argument similar to the one used to prove inequality \eqref{H1diff} shows that 
\begin{align*}
\left|L_{t}\left(-\Upsilon_{1}(x)\cdot t\right)\right|\le C_{\Sigma}\cdot t^{1/2}.
\end{align*}
\end{proof}

\begin{lemma}\label{FDR}
Suppose $f$ is $C^{5}\left(  \Sigma\right)  .$ \ Then%
\[
\Gamma_{t}(L_{t}f,f))(x)=\nabla\Delta_{g}f\cdot\nabla f+A_{1}^{\prime}(x,t;f)
\]
for
\[
\left|A_{1}^{\prime}(x,t;f)\right|\leq C(\Sigma,\|f\|_{C^{5}(\Sigma)})t.
\]
\end{lemma}

\begin{proof}
By Lemma \ref{Laplacef} we can write locally near $x_{0}$%
\[
L_{t}f(x)=\Delta_{g}f(x)+F_{1}(x;f)t-\Upsilon_{1}(x)t+\hat{F}_{3/2}(x,t;f).
\]
Noting that $\Delta_{g}f(x)-\Upsilon_{1}(x)t$ is $C^{3}(\Sigma)$ and $F_{1}(x)t$ is
$C^{1}(\Sigma).$ \ Using bilinearity,  we may compute
\begin{align}
\Gamma_{t}(L_{t}f,f)(x) &  =\Gamma_{t}\left(  \Delta_{g}f-\Upsilon
_{1}\cdot t,f\right)  (x)\label{Blue Dream}\\
&  +\Gamma_{t}\left(F_{1}(x;f)\cdot t,f\right)(x)\label{Northern Lights}\\
&  +\Gamma_{t}(\hat{F}_{3/2}(x,t;f),f))(x).\label{Krishna Kush}%
\end{align}
To begin, we may use Lemma \ref{smooth_local} to obtain \eqref{Blue Dream} in the following way: without loss of generality we may write

\[
\Delta_{g}f(y)-\Upsilon_{1}(y)\cdot t=\alpha(t,y)+\beta(t,y)
\]
with
\[
J^{2}\left(\alpha(t,x)\right)\bigg\rvert_{x=x_0}=J^{2}\Delta_{g}f(x_0)-tJ^{2}\Upsilon_{1}(x_0)
\]
as in (\ref{jetsagree}) and 
\begin{align*}
    \left |\beta(t,x)\right|\le C(\Sigma,\|f\|_{C^{5}(\Sigma)})\|x-x_0\|^{3}_{N}.
\end{align*}

\ Now we compute (\ref{Blue Dream}), again using
bilinearity
\[
\Gamma_{t}(\Delta_{g}f-\Upsilon_{1}(t,f))(x)=\Gamma_{t}(\alpha,f)(x)+\Gamma
_{t}(\beta,f)(x).
\]
The first term involves $C^{5}$ terms so we are able to apply Lemma \ref{cdc1}%
\begin{align*}
\Gamma_{t}(\alpha,f)\left(  x_{0}\right)   &  =\nabla\alpha\cdot\nabla
f(x_{0})+H_{1}(x_{0};\alpha,f)t-\Upsilon_{1}(x_{0})t+\hat{H}_{3/2}%
(x_0,t;\alpha,f)\\
&  =\nabla\Delta_{g}f\cdot\nabla f\left(  x_{0}\right)  -t\nabla\Upsilon
_{1}\cdot\nabla f(x_{0})\\
&  +H_{1}(x_{0};\alpha,f)t-\Upsilon_{1}(x_{0})t+\hat{H}_{3/2}(x_{0}%
,t;\alpha,f)\\
&  =\nabla\Delta_{g}f\cdot\nabla f(x_{0})+O(t;J^{1}f).
\end{align*}
We now estimate $\Gamma_t(\beta,f)$:
\begin{align*}
 \Gamma_t(\beta,f)=\frac{1}{\theta_{t}(x_0)}\int_{\Sigma}(\beta(t,y)-\beta(t,x_0))(f(y)-f(x_0))e^{-\frac{\|y-x_0\|^{2}_{N}}{2t}}d\mu(y)   
\end{align*}
and since $|\beta(t,y)-\beta(t,x_0)|\le C(\Sigma,\|f\|_{C^{5}(\Sigma)})\|x-y\|^{3}_{N}$ it follows that
\begin{align*}
  \left|\Gamma_{t}(\beta,f)(x)\right|\le C(\Sigma,\|f\|_{C^{5}(\Sigma)})\cdot t.
\end{align*}
Next we deal with (\ref{Northern Lights}): \ Recall (\ref{Cherry Thunder Fuck}%
)%
\begin{align}
\left\vert \Gamma_{t}(F_{1}(x)t,f)\right\vert  &  \leq2\left(  1-\Upsilon
_{1}(x)t+\hat{\Upsilon}_{2}(t,x)\right)  \cdot C\cdot t\cdot\left\Vert F_{1}(x)\right\Vert
_{\mathrm{Lip}}\left\Vert f\right\Vert _{\mathrm{Lip}}\\
&  \leq t\cdot C\cdot\Vert f\Vert _{C^{5}(\Sigma)}\cdot\left\Vert f\right\Vert
_{C^{1}(\Sigma)}%
\end{align}
recalling that definition of $F_{1}(x)$ \  (\ref{Purple Haze}) puts $F_{1}(x)$
in $C^{1}\left(  \Sigma\right).$ Finally, by similar reasoning as before (see computations near
(\ref{Wonder Kid}) and (\ref{Wunder Kind})) we can bound the final term
(\ref{Krishna Kush})) via
\begin{align*}
&  \left\vert \Gamma_{t}(\hat{F}_{3/2}(x,t;f),f))(x)\right\vert \\
&  \leq\frac{1}{t}\frac{(2\pi t)^{d/2}}{\theta_{t}(x)}\frac{1}{(2\pi t)^{d/2}%
}\int_{\Sigma}e^{-\frac{\Vert y-x_{0}\Vert_{N}^{2}}{2t}}\cdot C(\Sigma
,\|f\|_{C^{5}(\Sigma)})\cdot t^{3/2}\cdot\left\Vert f\right\Vert _{\mathrm{Lip}}\left\Vert y-x_{0}\right\Vert_{N}
d\mu(y)\\
&  \le C^{\prime}(\Sigma,\|f\|_{C^{5}(\Sigma)})\cdot t.
\end{align*}

\end{proof}

\begin{proof}[Proof of Proposition \ref{tconvergence}] The proof of \eqref{densityt} follows from Lemma \ref{theta_expand}, \eqref{Ltexpand} follows from Lemma \ref{Laplacef} and \eqref{cdct} is a consequence of Lemma \ref{Pink Floyd}. On the other hand, \eqref{gamma2approx} follows from combining Lemmas \ref{l319} and \ref{FDR}.
\end{proof}

The argument that has led to the proof of Proposition \eqref{tconvergence} has a fundamental takeaway, namely that at all stages of the approximation we may determine the rate of convergence to the target object (e.g. Carr\'{e} du Champ) by expanding the test functions in question and separating the resulting expansion into homogeneous terms and terms that contain remainders. In other words, the factors that appear in the different integrands in this article and that are independent of the test functions, i.e. the Gaussian kernel and the volume form, admit expansions that once combined with the expansions of the test functions can only help the rate of convergence to the target object. This observation leads to a practical formulation of our argument that will be illustrated with the following corollary.  
 
\begin{corollary}\label{Cor4}
Let $\Sigma^{d}$ be a close embedded submanifold of $\R^{N}$ and let $g$ be the induced metric. Let $f$ be a function of class  $C^{4}(\Sigma)$. We have the following estimates
\begin{align}
    L_{t}f(x)&=\Delta_{g}f(x)+O(t;\Sigma,J^{4}f(x))\label{Lap4}\\
    \Gamma_{t}(f,h)(x) &= \langle\nabla f,\nabla h\rangle_{g}(x)+O(t;\Sigma,J^{3}f(x),J^{3}h(x))\label{CDC4}\\
    L_{t}(\Gamma_t(f,h))(x) &= \Delta_{g}\langle\nabla f,\nabla h\rangle_{g}(x)+O(t^{1/2};\Sigma,J^{4}f(x),J^{4}h(x))\label{LCDC4}\\
    \Gamma_{t}(L_t f,h)(x) &=\langle\Delta_{g}\nabla f,h\rangle_{g}(x)+O(t^{1/2};\Sigma,J^{4}f(x),J^{2}h(x)).\label{GammaL4}
\end{align}
\end{corollary}
\begin{proof}
The proof of Corollary \ref{Cor4} follows from combining a schematic formulation of the bounds \eqref{cdct}-\eqref{gamma2approx} when one expands the test functions $f(x)$ and $h(x)$ only up to fourth order, in the proof of Lemma \ref{BigNasty}. Note that this approach still relies on the estimate \eqref{order8} and expansion \eqref{l98}. More precisely, for \eqref{Lap4} we let again $\tilde{f}(x,z)=f(\exp_{x}(z))$  and consider the expansion 
\begin{align*}
    \tilde{f}(x,z) =\tilde{f}(x,0)+D_{z}^{1}\tilde{f}(x,0)[z]+D_{z}^{2}\tilde{f}(x,0)[z]+D_{z}^{3}\tilde{f}(x,0)[z]+\rho_{4}(\tilde{f},x)[z]
\end{align*}
then from Lemma \ref{BigNasty} 
\begin{align*}
    L_{t}f(x)=\Delta_{g}f(x)+O(t;\Sigma,J^{4}f(x)).
\end{align*}
For \eqref{CDC4} expand both $f$ and $h$ to fourth order
\begin{align*}
\tilde{f}(x,z) &=\tilde{f}(x,0)+D_{z}^{1}\tilde{f}(x,0)[z]+D_{z}^{2}\tilde{f}(x,0)[z]+D_{z}^{3}\tilde{f}(x,0)[z]+\rho_{4}(\tilde{f},x)[z]\\
\tilde{h}(x,z) &=\tilde{h}(x,0)+D_{z}^{1}\tilde{h}(x,0)[z]+D_{z}^{2}\tilde{h}(x,0)[z]+D_{z}^{3}\tilde{h}(x,0)[z]+\rho_{4}(\tilde{h},x)[z]
\end{align*}
so that schematically we have
\begin{align*}
(\tilde{f}(x,z)-\tilde{f}(x,0))&(\tilde{h}(x,z)-\tilde{h}(x,0))=D_{z}^{1}\tilde{f}(x,0)[z]\otimes D_{z}^{1}\tilde{h}(x,0)[z]\\
&+D_{z}^{1}\tilde{f}(x,0)[z]\otimes D_{z}^{2}\tilde{h}(x,0)[z]+D_{z}^{2}\tilde{f}(x,0)[z]\otimes D_{z}^{1}\tilde{h}(x,0)[z]\\
&+\left(D^{3}_{z}\tilde{f}(x,0)[z]\otimes D^{1}_{z}\tilde{h}(x,0)[z]+D^{1}_{z}\tilde{f}(x,0)[z]\otimes D^{3}_{z}\tilde{h}(x,0)[z]\right)\\
&+\left(D^{2}_{z}\tilde{f}(x,0)[z]\otimes D^{2}_{z}\tilde{h}(x,0)[z]+\sum_{\alpha=5}^{8}\sum_{\beta+\gamma =\alpha}D^{\beta}_{z}\tilde{f}[z]*D^{\gamma}_{z}\tilde{h}[z]\right)      
\end{align*}
where a term of the form $D^{k}\tilde{f}*D^{l}\tilde{h}[z]$ represents the product of a homogeneous term and a remainder term involving derivatives of order $k$ on $\tilde{f}$ and derivatives of order $l$ on $\tilde{h}$, and satisfying 
\begin{align*}
   D^{k}_{z}\tilde{f}*D^{l}_{z}\tilde{h}[z] = O(\|z\|^{k+l}_{d};\Sigma,J^{k}\tilde{f},J^{l}\tilde{h}).
\end{align*}
Note also that each term of the form 
\begin{align*}
    D_{z}^{k}\tilde{f}(x,0)[z]\otimes D_{z}^{l}\tilde{h}(x,0)[z]
\end{align*}
is homogeneous of degree $k+l$. We therefore have
\begin{align*}
    \Gamma_{t}(f,h)(x) = \langle\nabla f,\nabla h\rangle_{g}(x)+t\cdot P(J^{3}f,J^{3}h)(x)+O(t^{3/2};\Sigma,J^4f(x),J^4h(x)).
\end{align*}
Let us write now write (while implicitly defining discrepancy functions)
\begin{align}
    L_{t}f(x) &= \Delta_{g}f(x)+Q_1(t,x,J^{4}f)\label{LapSchem}\\
    \Gamma_{t}(f,h)(x) &= \langle \nabla f,\nabla h\rangle_{g}(x)+t\cdot P(J^{3}f,J^{3}h)(x)+O(t^{3/2};\Sigma,J^4f(x),J^4h(x)).\label{CDCSchem}
\end{align}
From \eqref{LapSchem} we have 
\begin{align*}
    (L_{t}f(y)&-L_{t}f(x))(h(y)-h(x))\\
    &=\left(\Delta_{g}f(y)-\Delta_{g}f(x)+Q_1(t,y,J^{4}f)-Q_1(t,x,J^{4}f)\right)(h(y)-h(x))
\end{align*}
and if we let $\upsilon(x)=\Delta_{g}f(x)$ we have 
\begin{align}
    \upsilon(y)-\upsilon(x) &= D^{1}_{z}\tilde{\upsilon}(x,0)[z]+\rho_{2}(\tilde{\upsilon},x)[z]\\
    h(y)-h(x) &=D^{1}_{z}\tilde{h}(x,0)[z]+\rho_{2}(\tilde{h},x)[z]  
\end{align}
and we write
\begin{align*}
 (L_{t}f(y)&-L_{t}f(x))(h(y)-h(x))\\
 &= D^{1}_{z}\tilde{\upsilon}(x,0)[z]\otimes D^{1}_{z}\tilde{h}(x,0)[z]+D^{1}_{z}\tilde{\upsilon}*D^{2}_{z}\tilde{h}[z]+D^{2}_{z}\tilde{\upsilon}*D^{1}_{z}\tilde{h}[z]\\
 &+(Q_1(t,y,J^{4}f)-Q_1(t,y,J^{4}f))(D^{1}_{z}\tilde{h}(x,0)[z]+\rho_{2}(\tilde{h},x)[z]).
\end{align*}
We conclude that 
\begin{align*}
    \Gamma_{t}(L_{t}f,h)(x) =\langle\nabla\Delta_{g}f,\nabla h\rangle(x) +O(t^{1/2};\Sigma,J^{4}f(x),J^{2}h(x)). 
\end{align*}
Finally, we have 
\begin{align*}
 \Gamma_{t}(f,h)(y)-\Gamma_{t}(f,h)(x)&= \langle \nabla f,\nabla h\rangle_{g}(y)-\langle \nabla f,\nabla h\rangle_{g}(x)\\
 &+ t\cdot (P(J^{3}f,J^{3}h)(y)-P(J^{3}f,J^{3}h)(x))+ O(t^{3/2};\Sigma,J^4f,J^4h)
\end{align*}
and if now $\upsilon(x) =\langle \nabla f,\nabla h\rangle_{g}(x)$, then we may write
\begin{align*}
    \upsilon(y)-\upsilon(x)=D^{1}_{z}\tilde{v}(x,0)[z]+D^2_{z}\tilde{\upsilon}(x,0)[z]+\rho_{3}(\tilde{\upsilon},x)[z]
\end{align*}
which implies
\begin{align*}
 \Gamma_{t}(f,h)(y)-\Gamma_{t}(f,h)(x)&= D^1_{z}\tilde{\upsilon}(x,0)[z]+D^2_{z}\tilde{\upsilon}(x,0)[z]+\rho_{3}(\tilde{\upsilon},x)[z]\\
 &+ t\cdot (P(J^{3}f,J^{3}h)(y)-P(J^{3}f,J^{3}h)(x))\\
 &+ O(t^{3/2};\Sigma,J^4f,J^4h)
\end{align*}
and using an inequality similar to \eqref{geodesic_estimate} we find that for $y$ in a geodesic ball of the form $B_{\varepsilon}(x)$ where $\varepsilon$ is as in \eqref{epschoice}%
\begin{align*}
\left|(P(J^{3}f,J^{3}h)(y)-P(J^{3}f,J^{3}h)(x))\right|\le C(\Sigma,\|f\|_{C^{4}(\Sigma)},\|h\|_{C^{4}(\Sigma)})d_{\Sigma}(x,y).
\end{align*}
We conclude that
\begin{align}
L_{t}(\Gamma_{t}(f,h))(x) = \Delta_{g}\langle \nabla f,\nabla h\rangle_{g}(x) +O(t^{1/2};\Sigma,J^{4}f(x),J^{4}h(x)),
\end{align}
which proves the corollary.
\end{proof}

\begin{proof}[Proof of Theorem \ref{extrinsic bias}] Theorem \ref{extrinsic bias} follows easily from \eqref{LCDC4} and \eqref{GammaL4}. 
\end{proof}

\section{Non-uniform density} \label{miscellaneous}

\subsection{Bias for Smooth Metric Measure Spaces with Density}
\bigskip We do not attempt to prove Theorem \ref{intrinsic bias} in full detail in order to avoid making the presentation of the article tedious. Instead we provide a basic plotline.  

\begin{itemize}
\item First, our expansion of integration density with the new weight
$e^{-\vartheta(x)}d\mu=e^{-\vartheta}d\mathrm{vol}_{\Sigma}$ gives a similar expansion%
\[
\frac{\theta_{t}(x)}{(2\pi t)^{d/2}}=e^{-\vartheta(x)}\left(  1+S_{1}(x)t+O(t^{2};\Sigma
)\right)
\]
where $S_{1}$ is a smooth term that depends on geometric data and also $\vartheta.$

\item Using the expansion%
\begin{align}
\left(  \frac{1}{1+\eta}\right)  ^{\alpha}=1-\alpha\eta+\alpha\left(
\alpha+1\right)  \eta^{2}+O\left(  \eta^{3};\Sigma\right)\label{inverse_alpha_theta}
\end{align}

for $\left\vert \eta\right\vert <\frac{1}{2}$ we  can expand and collect
\[
\frac{1}{\left[  \theta_{t}(x)\right]  ^{\alpha}}=(2\pi t)^{-\alpha
d/2}e^{\alpha \vartheta(x)}\left(  1+S_{2}(x)t+O(t^{2};\Sigma)\right)
\]
for some $S_{2}$ depending on geometric data, $\alpha\,\ $and $\vartheta.$ \ (This
requires $t$ chosen small, which we may assume.)

\item Inserting \eqref{inverse_alpha_theta} into the expression for $\theta_{t,\alpha}(x)$ defined in (\ref{alpha1}), we have
\begin{equation}
\theta_{t,\alpha}(x)=(2\pi t)^{\left(  1-\alpha\right)  d/2}e^{\alpha
\vartheta(x)}\left(  1+S_{3}(x)t+O(t^{2};\Sigma)\right)
.\label{todo es mentira en esta mundo}%
\end{equation}
for some $S_{2}$ depending on geometric data, $\alpha\,\ $and $\vartheta.$

\item Now turning to the expression $L_{t}^{\alpha}f(x)$ (\ref{alpha2}) we
apply first Lemma \ref{BigNasty} to
\begin{align*}
f(y)=e^{\left(  \alpha-1\right)  \vartheta(y)}\left(  1+S_{3}(y)t+O(t^{2};\Sigma)\right)
\left(  f(y)-f(x)\right)
\end{align*}
we see that
\begin{align*}
& \int_{\Sigma} e^{-\frac{\|x-y\|^2_{N}}{2t}}\frac{1}{\left[  \theta_{t}(y)\right]
^{\alpha}}\left(  f(y)-f(x)\right)  d\mu(y)\\
& =(2\pi t)^{\left(  1-\alpha\right)  d/2}e^{\left(  \alpha-1\right)
\vartheta(x)}\left(  \left(  \alpha-1\right)  \langle\nabla \vartheta,\nabla f\rangle_{g}(x)t+\frac{1}%
{2}\Delta_{g} f(x)t+O(t^{2};\Sigma)\right)
\end{align*}
which combines with (\ref{todo es mentira en esta mundo}), (also using Lemma
\ref{basic_monomials} as in (\ref{la verdad})) to give\
\[
L_{t}^{\alpha}f(x)=2\left(  \alpha-1\right)  \langle \nabla \vartheta, \nabla f(x) \rangle_g+\Delta_{g}
f(x)+S_{4}(x)t+O(t^{3/2};\Sigma)
\]
where now $S_{4}$ is an expression in the derivatives of $f$ as
in Lemma \ref{Laplacef}. This outlines a weighted version of Lemma \ref{Laplacef}. The arguments for lemmas  \ref{cdc1} through \ref{FDR} are similar. 
\end{itemize}

\section{Convergence of Coarse Ricci to Actual Ricci on Smooth
submanifolds}

\label{riccisubmanifold}
Before proving Theorem \ref{HereWeRecoverRicci}, we will derive several observations for linear functions defined in $\R^{N}$. First, let X be a vector in $\R^{N}$. The vector $X$ has both tangential and normal components to $\Sigma^{d}$ that we will denote the tangential and normal components of $X$ by $X^{T}$ and $X^{\perp}$ respectively. On the other hand, let $\nabla$ denote covariant differentiation with respect to the Euclidean metric in the ambient space $\R^{N}$, and let $\{n_{1},\ldots,n_{N-d}\}$ denote an orthonormal basis of vector fields for the normal space to $\Sigma$ in a neighborhood of $x$ in $\Sigma$. In this case, given vectors $X,Y$ in $T_{x}\Sigma$, the components of the second fundamental form $\mathrm{II}(X,Y)$ with respect to this basis can be identified with the scalars 
\begin{align*}
    \langle Y,\nabla_{X}n_{1}\rangle,\ldots,\langle Y,\nabla_{X}n_{N-d}\rangle
\end{align*}
in the following sense: in principle, the second fundamental form is given by $(\nabla_{X}Y)^{\perp}$, which is represented in these coordinates by 
\begin{align*}
    \sum_{l=1}^{N-d}\langle (\nabla_{X}Y)^{\perp},n_{l}\rangle n_{l}= -\sum_{l=1}^{N-d}\langle Y,\nabla_{X}n_{l}\rangle n_{l}.
\end{align*}
We will denote the $l$-th component of the second fundamental form by 
\begin{align*}
    \mathrm{II}_{l}(X,Y) =\langle\nabla_{X}n_{l},Y\rangle,
\end{align*}
which strictly speaking has a sign discrepancy with the second fundamental form, but this discrepancy will not be important for the analysis below. With this in mind, we have the following lemma:
\begin{lemma}\label{linear}
Let $f:\R^{N}\rightarrow\R$ be a linear function written as $f(x)=\langle\zeta,x\rangle$ for some fixed $\zeta$ in $\R^{N}$ and where $\langle\cdot,\cdot\rangle$ is the Euclidean inner product in $\R^{N}$. Let $x$ be a point in $\Sigma^{d}$ and let $\{n_{1},\ldots,n_{N-d}\}$ be an orthonormal basis of vector fields for the normal space to $\Sigma$ near $x$. Finally, let $\{\mathrm{II}_{}\}_{l=1,\ldots,N-d}$ be the components of the second fundamental form $\mathrm{II}$ with respect to this basis. Then, if we let $\overline{\nabla}$ denote covariant differentiation with respect to the induced metric $g_{\Sigma}$ and if $X,Y,Z,U$ are vectors in $ T_{x}\Sigma$ we have    
\begin{align}
    \overline{\nabla}_{Y}\overline{\nabla}_{X}f = -\sum_{l=1}^{N-d}\mathrm{II}_{l}(X,Y)\langle \zeta,n_{l}\rangle,\label{SecondDerivative}
\end{align}
and also
\begin{align}    \overline{\nabla}_{Z}\overline{\nabla}_{Y}\overline{\nabla}_{X}f &=-\sum_{l=1}^{N-d}\left(\overline{\nabla}_{Z}\mathrm{II}_{l}(X,Y)\langle\zeta,n_{l}\rangle+\mathrm{II}_{l}(Z,\zeta^{T})\otimes\mathrm{II}_{l}(X,Y)\right)\label{ThirdDerivative}
\end{align}
where $\overline{\nabla}_{Z}\mathrm{II}_{l}(X,Y)$ is the covariant derivative of the tensor $\mathrm{II}_{l}(X,Y)$ as a tensor in $S^{2}\left(T^{*}\Sigma\otimes T^{*}\Sigma\right)$. Finally, for fourth order derivatives we have the following identity
\begin{align*}
    &\overline{\nabla}_{U}\overline{\nabla}_{Z}\overline{\nabla}_{Y}\overline{\nabla}_{X}f = -\sum_{l=1}^{N-d}\left(\overline{\nabla}_{U}\overline{\nabla}_{Z}\mathrm{II}_{l}(X,Y)\langle\zeta,n_{l}\rangle+\mathrm{II}_{l}(\zeta^{T},U)\otimes\overline{\nabla}_{Z}\mathrm{II}_{l}(X,Y)\right)\\
    &+\sum_{l=1}^{N-d}\left(\left(\sum_{m=1}^{N-d}\langle n_{m},\zeta\rangle\langle W_{l}(Z),W_{m}(U)\rangle-\langle\zeta,\overline{\nabla}W_{l}(U,Z)\rangle\right)\otimes \mathrm{II}_{l}(X,Y)\right)\\
    &+\sum_{l=1}^{N-d}\left(\mathrm{II}_{l}(Z,\zeta^{T})\otimes\overline{\nabla}_{U}\mathrm{II}_{l}(X,Y)\right)
\end{align*}
where $W_{l}:T_{x}\Sigma\rightarrow T_{x}\Sigma$ is the Weingarten map given by $W_{l}(x)=\nabla_{X}n_{l}$, and $\overline{\nabla}W_{l}(U,Z)$ is the covariant derivative of $W_{l}(Z)$ as a $(1,1)$-tensor.
\end{lemma}

\begin{proof}
Let $X$ be a vector in $T_{x}\Sigma$, then we clearly have 
\begin{align*}
    \overline{\nabla}_{X}f=\langle\zeta,X\rangle = \langle\zeta^{T},X\rangle,
\end{align*}
and for the second covariant derivatives we have 
\begin{align*}
    \overline{\nabla}_{Y}\overline{\nabla}_{X}f &= \nabla_{Y}\langle\zeta,X\rangle -\langle\zeta,\overline{\nabla}_{Y}X\rangle =\langle\zeta,\nabla_{Y}X\rangle-\langle\zeta,\overline{\nabla}_{Y}X\rangle\\
    &=\langle\zeta,(\nabla_{Y}X)^{\perp}\rangle =\sum_{l=1}^{N-d}\langle\nabla_{Y}X,n_{l}\rangle\langle n_{l},\zeta\rangle\\
    &= -\sum_{l=1}^{N-d}\langle X,\nabla_{Y} n_{l}\rangle\langle \zeta,n_{l}\rangle\\
    &=-\sum_{l=1}^{N-d}\mathrm{II}_{l}(X,Y)\langle\zeta,n_{l}\rangle,
\end{align*}
which proves \eqref{SecondDerivative}.
For the third derivatives in \eqref{ThirdDerivative} we have
\begin{align*}
   \overline{\nabla}_{Z}\overline{\nabla}_{Y}\overline{\nabla}_{X}f &=-\sum_{l=1}^{N-d}\left(\overline{\nabla}_{Z}\mathrm{II}_{l}(X,Y)\langle\zeta,n_{l}\rangle+\overline{\nabla}_{Z}\langle\zeta,n_{l}\rangle\otimes\mathrm{II}_{l}(X,Y)\right)\\
   &=-\sum_{l=1}^{N-d}\left(\overline{\nabla}_{Z}\mathrm{II}_{l}(X,Y)\langle\zeta,n_{l}\rangle+\langle\zeta,\nabla_{Z}n_{l}\rangle\otimes\mathrm{II}_{l}(X,Y)\right)\\
   &=-\sum_{l=1}^{N-d}\left(\overline{\nabla}_{Z}\mathrm{II}_{l}(X,Y)\langle\zeta,n_{l}\rangle+\langle\zeta^{T},\nabla_{Z}n_{l}\rangle\otimes\mathrm{II}_{l}(X,Y)\right)\\
   &=-\sum_{l=1}^{N-d}\left(\overline{\nabla}_{Z}\mathrm{II}_{l}(X,Y)\langle\zeta,n_{l}\rangle+\mathrm{II}_{l}(\zeta^{T},Z)\otimes\mathrm{II}_{l}(X,Y)\right).
\end{align*}
The proof for the fourth derivatives follows a similar strategy with the exception of the terms that contain covariant derivatives of the form $\overline{\nabla}_{U}\mathrm{II}_{l}(\zeta^{T},Z)$,which we now compute: observe that $\overline{\nabla}_{U}\mathrm{II}_{l}(\zeta^{T},Z)$ is the covariant derivative of the 1-form $\eta(Z)=\mathrm{II}_{l}(\zeta^{T},Z)$ which we compute as 
\begin{align*}
    \overline{\nabla}_{U}\eta(Z) &= U(\eta(Z))-\eta(\overline{\nabla}_{U}Z)\\
    &=U\langle\zeta,\nabla_{Z}n_{l}\rangle-\langle\zeta,\nabla_{\overline{\nabla}_{U}Z}n_{l}\rangle\\
    &=\langle\zeta,\nabla_{U}W_{l}(Z)\rangle-\langle\zeta,W_{l}(\overline{\nabla}_{U}Z)\rangle,
\end{align*}
and we may write
\begin{align*}
 \nabla_{U}W_{l}(Z)&=\left(\nabla_{U}W_{l}(Z)\right)^{T}+\left(\nabla_{U}W_{l}(Z)\right)^{\perp}\\
 &=\overline{\nabla}_{U}W_{l}(Z)+\left(\nabla_{U}W_{l}(Z)\right)^{\perp}.
 \end{align*}
In order to compute $\left(\nabla_{U}W_{l}(Z)\right)^{\perp}$, we write
\begin{align*}
    \left(\nabla_{U}W_{l}(Z)\right)^{\perp} &= \sum_{m=1}^{N-d}\langle\nabla_{U}W_{l}(Z),n_{m}\rangle n_{m}\\
    & = -\sum_{m=1}^{N-d}\langle W_{l}(Z),\nabla_{U}n_{m}\rangle n_{m}\\
    & = -\sum_{m=1}^{N-d}\langle W_{l}(Z),W_{m}(U)\rangle n_{m}.
\end{align*}
We have shown 
\begin{align*}
\overline{\nabla}_{U}\eta(Z) =-\sum_{m=1}^{N-d}\langle W_{l}(Z),W_{m}(U)\rangle\langle\zeta,n_{m}\rangle+\langle\overline{\nabla}_{U}W_{l}(Z)-W_{l}(\overline{\nabla}_{U}Z),\zeta\rangle
\end{align*}
and observe that $\overline{\nabla}_{U}W_{l}(Z)-W_{l}(\overline{\nabla}_{U}Z) = \overline{\nabla}W_{l}(U,Z)$ when one views $W_{l}(Z)$ as a $(1,1)$-tensor. This completes the proof.
\end{proof}

We have proved the following
\begin{corollary}\label{corLinear}
Let $f:\R^{N}\rightarrow\R$ be a linear function written as $f(x)=\langle\zeta,x\rangle$. Then for $2\le k\le 4$ we have
\begin{align*}
    \|\nabla^{k}_{g_{\Sigma}}f\|_{C^{0}(\Sigma)}\le C_{k}\|\zeta\|_{N}
\end{align*}
where $C_k$ depends on $C^{k-2}(\Sigma)$ bounds on the second fundamental form $\mathrm{II}$.
\end{corollary}

We now prove Theorem \ref{HereWeRecoverRicci}:


\begin{proof}
[Proof of Theorem \ref{HereWeRecoverRicci}] In the discussion that follows $g$ represents the metric $g_{\Sigma}$. We also write $\nabla_{g}f$ instead of $\nabla f$ to stress that we are performing differentiation on the submanifold and not on the ambient space. Our goal is to show that
\[
|\mathrm{Ric}_{g}(\gamma^{\prime}(0),\gamma^{\prime}(0))-\mathrm{RIC}_{L_{t}%
}(x,\gamma(s))|\leq C_{1}t^{1/2}+C_{2}s.
\]
First, note that letting
\[
f_{s}=F_{x,\gamma(s)}(\cdot)=\left\langle \frac{\gamma(s)-x}{\left\Vert
\gamma(s)-x\right\Vert_{N} },\cdot\right\rangle
\]
and
\[
f_{0}=\langle\gamma^{\prime}(0),\cdot\rangle
\]
we have
\begin{align}
\left\vert \mathrm{RIC}_{L_{t}}(x,\gamma(s))-\mathrm{Ric}(\gamma^{\prime
}(0),\gamma^{\prime}(0))\right\vert  &  =\left\vert \Gamma_{2}(L_{t}%
,f_{s},f_{s})(x)-\mathrm{Ric}(\gamma^{\prime}(0),\gamma^{\prime}%
(0))\right\vert \\
&  =\left\vert
\begin{array}
[c]{c}%
\left(\Gamma_{2}(L_{t},f_{s},f_{s})(x)-\Gamma_{2}(\triangle_{g},f_{s},f_{s})(x)\right)\\
+\left(\Gamma_{2}(\triangle_{g},f_{s},f_{s})(x)-\Gamma_{2}(\triangle_{g},f_{0}%
,f_{0})(x)\right)\\
+\left(\Gamma_{2}(\triangle_{g},f_{0},f_{0})(x)-\mathrm{Ric}(\gamma^{\prime
}(0),\gamma^{\prime}(0))\right)
\end{array}
\right\vert \\
&  \leq C(\Sigma,\|f_{s}\|_{C^{4}(\Sigma)})t^{1/2}+C(\Sigma,\|\mathrm{II}\|_{C^{1}(\Sigma)})s\\
&\le C(\Sigma,\|\mathrm{II}\|_{C^{2}(\Sigma)})t^{1/2}+C(\Sigma,\|\mathrm{II}\|_{C^{1}(\Sigma)})s.
\end{align}
Here we have used the following facts:
\begin{itemize}
\item First,
\begin{align*}
\left\vert \Gamma_{2}(L_{t},f_{s},f_{s})(x)-\Gamma_{2}(\triangle_{g}%
,f_{s},f_{s})(x)\right\vert \leq C(\Sigma,\|f_{s}\|_{C^{4}(\Sigma)})t^{1/2}
\end{align*}
by Lemma \ref{l319}  and \ref{FDR}. Since $f_{s}$ is a 1-parameter family of linear functions defined on $\R^{N}$, we obtain from Lemma \ref{linear} the bound in absolute value
\begin{align*}
\left\vert \Gamma_{2}(L_{t},f_{s},f_{s})(x)-\Gamma_{2}(\triangle_{g}%
,f_{s},f_{s})(x)\right\vert &\leq C(\Sigma,\|\mathrm{II}\|_{C^{2}(\Sigma)})\|\zeta_{s}\|_{N}t^{1/2}\\
& =C(\Sigma,\|\mathrm{II}\|_{C^{2}(\Sigma)})t^{1/2} 
\end{align*}
where 
\begin{align*}
\zeta_{s}=\frac{\gamma(s)-x}{\left\vert\left\vert\gamma(s)-x\right\vert\right\vert_{N}}.
\end{align*}
\item Second, a straightforward computation yields that for any two functions
$f_{0},f_{s}$
\begin{align*}
 \Gamma_{2}(\Delta_{g},f_{s},f_{s})-\Gamma_{2}(\Delta_{g},f_{0},f_{0})&= \left\langle\left(\Delta_{g}\nabla_{g}-\nabla_{g}\Delta_{g}\right)(f_{s}-f_{0}),\nabla_{g}f_{s}\right\rangle_{g}\\
 &+\left\langle(\Delta_{g}\nabla_{g}-\nabla_{g}\Delta_{g})(f_{s}-f_{0}),\nabla_{g}f_{0}\right\rangle_{g}\\
 &+\langle\nabla^{2}_{g}(f_{s}-f_{0}),\nabla^{2}_{g}f_{s}\rangle_{g}\\
 &+\langle\nabla^{2}_{g}(f_{s}-f_{0}),\nabla^{2}_{g}f_{0}\rangle_{g}
\end{align*}
and therefore, we have for every $x$ the inequality
\begin{align*}
\left| \Gamma_{2}(\Delta_{g},f_{s},f_{s})(x)-\Gamma_{2}(\Delta_{g},f_{0},f_{0})(x)\right|&\leq C_{d}\left\|\nabla^{3}_{g}(f_{s}-f_{0})\right\|_{C^{0}(\Sigma)}\left(\left\|\nabla_{g}f_{s}\right\|_{C^{0}(\Sigma)}+\left\|\nabla_{g}f_{0}\right\|_{C^{0}(\Sigma)}\right)\\
&+\left\|\nabla^{2}_{g}(f_{s}-f_{0})\right\|_{C^{0}(\Sigma)}\left(\left\|\nabla^{2}_{g}f_{s}\right\|_{C^{0}(\Sigma)}+\left\|\nabla^{2}_{g}f_{0}\right\|_{C^{0}(\Sigma)}\right)
\end{align*}
where $C_{d}$ is a dimensional constant, and by Corollary \ref{corLinear} we have
\begin{align*}
 \left| \Gamma_{2}(\Delta_{g},f_{s},f_{s})(x)-\Gamma_{2}(\Delta_{g},f_{0},f_{0})(x)\right|&\le C(\Sigma,\|\mathrm{II}\|_{C^{1}(\Sigma)})\|\zeta_{s}-\gamma^{\prime}(0)\|_{N}\left(\|\zeta_{s}\|_{N}+\|\gamma^{\prime}(0)\|_{N}\right)\\
 &+C(\Sigma,\|\mathrm{II}\|_{C^{0}(\Sigma)})\|\zeta_{s}-\gamma^{\prime}(0)\|_{N}\left(\|\zeta_{s}\|_{N}+\|\gamma^{\prime}(0)\|_{N}\right)
\end{align*}
and note that 
\begin{align*}
    \|\zeta_{s}-\gamma(0)\|_{N}\le M\cdot s
\end{align*}
where $M$ is a bound on the second derivatives of $\gamma(s)$ for $s$ near $0$. Note that the constant $M$ is in general comparable to the a bound on the absolute value of the curvature of the curve $\zeta_{s}$ near $s=0$. We therefore have

\begin{align*}
  \left| \Gamma_{2}(\Delta_{g},f_{s},f_{s})(x)-\Gamma_{2}(\Delta_{g},f_{0},f_{0})(x)\right|&\le C\left(\Sigma,\|\mathrm{II}\|_{C^{1}(\Sigma)}\right)\cdot M\cdot s.    
\end{align*}
Observe that we may avoid first covariant derivatives of the second fundamental form in the above constant by writing 
\begin{align*}
\left\langle\left(\Delta_{g}\nabla_{g}-\nabla_{g}\Delta_{g}\right)(f_{s}-f_{0}),\nabla_{g}f_{s}\right\rangle_{g} &= \mathrm{Ric}_{g}(\nabla_{g}(f_{s}-f_{0}),\nabla_{g}f_{s})\\
&+\mathrm{Ric}_{g}(\nabla_{g}(f_{s}-f_{0}),\nabla_{g}f_{0}),
\end{align*}
but in this case we would have to include bounds on the Ricci curvature in the constant $C$. 

\item Finally, because
\begin{align*}
\nabla_{g} f_{0}=(\gamma^{\prime}(0))^{T} = \gamma^{\prime}(0),
\end{align*}
we have by the Bochner formula%
\begin{equation}
\Gamma_{2}(\triangle_{g},f_{0},f_{0})(x)-\mathrm{Ric}_{g}(\gamma^{\prime
}(0),\gamma^{\prime}(0))=\left\Vert \nabla_{g}^{2}f_{0}\right\Vert ^{2}_{g}.
\label{hesscoord}%
\end{equation}
And using that $\gamma^{\prime}(0)$ is tangent to $\Sigma$ at $x$ we conclude from \eqref{SecondDerivative} in Lemma \ref{linear} that (\ref{hesscoord}) vanishes.
\end{itemize}
\end{proof}

\appendix

\section{Integrals of monomials}\label{appendixA}

We have encountered quite a few integrals of products of monomials and
Gaussian functions on balls centered at the origin of $\mathbb{R}^{d}$. It is
important to point out that these integrals will be computed with respect to
the Lebesgue measure of the Euclidean metric. We now explain how to compute
these integrals in a systematic way

\begin{lemma}
\label{basic_monomials} Let $z=(z_{1},\ldots,z_{d})\in\mathbb{R}^{d}$ and
consider a subset $\{i_{1},\ldots,i_{p}\}$ of $\{1,\ldots,d\}$. Given positive
integers $\alpha_{i_{1}},\ldots,\alpha_{i_{p}}$ we have
\begin{align}
\frac{1}{(2\pi)^{d/2}}\int_{\mathbb{R}^{d}}e^{-\frac{\|z\|_{d}^{2}}{2}%
}z^{\alpha_{i_{1}}}_{i_{1}}\cdot\ldots\cdot z^{\alpha_{i_{p}}}_{i_{p}%
}dz=\omega_{\alpha_{i_{1}},\ldots,\alpha_{i_{p}}} \label{intRd}%
\end{align}
where

\begin{itemize}
\item $\alpha=\alpha_{i_{1}}+\ldots+\alpha_{i_{p}}$

\item $\omega_{\alpha_{i_{1}},\ldots,\alpha_{i_{p}}}$ is defined by
\begin{align*}
\omega_{\alpha_{i_{1}}\ldots\alpha_{i_{p}}}=\left\{
\begin{array}
[c]{cc}%
0 & ~\text{if any}~\alpha_{i_{j}}~\text{is odd},~j=1,\ldots,p\\
\displaystyle{\prod_{j=1}^{p}}(2k_{i_{j}}-1)!! & ~\text{if}~\alpha_{i_{j}%
}=2k_{i_{j}}~\text{for}~j=1,\ldots,p
\end{array}
\right.
\end{align*}

\end{itemize}

In the above description, the number $(2k_{i_{j}}-1)!!$ is the \emph{odd
factorial}, that is the product of all odd numbers $1,3,\ldots,2k_{i_{j}}-1$.
\end{lemma}

A similar computation that we will need is
\[
\int_{\mathbb{R}^{d}}\Vert z\Vert_{d}^{k}e^{-\frac{\Vert z\Vert_{d}^{2}}{2}%
}dz=\mathrm{vol}\left(  S^{d-1}\right)  2^{\frac{k+d-2}{2}}\Gamma\left(
\frac{k+d}{2}\right)
\]
where $\Gamma\left(  \frac{k+d-1}{2}\right)  $ denotes the Gamma function and
$\mathrm{vol}\left(  S^{d-1}\right)  $ denotes the volume of the sphere
$S^{d-1}$ with respect to the round metric. Observe that
\[
\mathrm{vol}(S^{d-1})=\frac{(2\pi)^{d/2}}{2^{(d-2)/2}\Gamma(d/2)}%
\]
which implies
\[
\int_{\mathbb{R}^{d}}\Vert z\Vert_{d}^{k}e^{-\frac{\Vert z\Vert_{d}^{2}}{2}%
}dz=(2\pi)^{d/2}2^{k/2}\frac{\Gamma\left(  \frac{k+d}{2}\right)  }%
{\Gamma(\frac{d}{2})}%
\]
and by scaling we have the identity
\[
\frac{1}{(2\pi t)^{d/2}}\int_{\mathbb{R}^{d}}\Vert z\Vert_{d}^{k}%
e^{-\frac{\Vert z\Vert_{d}^{2}}{2t}}dz=t^{k/2}2^{k/2}\frac{\Gamma
((k+d)/2)}{\Gamma(d/2)}.
\]
Notice also that using $z=\sqrt{2}w$ we can compute
\begin{align}
\int_{\mathbb{R}^{d}}\Vert z\Vert_{d}^{k}e^{-\frac{\Vert z\Vert_{d}^{2}}{4t}%
}dz &  =\int_{\mathbb{R}^{d}}\Vert\sqrt{2}w\Vert_{d}^{k}e^{-\frac{\Vert
w\Vert_{d}^{2}}{2t}}\left(  \frac{1}{\sqrt{2}}\right)  ^{d}dw\nonumber\\
&  =(2\pi t)^{d/2}t^{k/2}2^{\left(  2k-d\right)  /2}\frac{\Gamma
((k+d)/2)}{\Gamma(d/2)}.\label{rescaled}%
\end{align}
From the above we observe that if $f\left(  z\right)  =O\left(  \left\Vert
z\right\Vert _{d}^{k}\right)  $ and $f(z)$ has polynomial growth as $\Vert
z\Vert_{d}\rightarrow\infty$, then
\begin{align}
\left\vert \int_{\mathbb{R}^{d}}f\left(  z\right)  e^{-\frac{\Vert z\Vert
_{d}^{2}}{2t}}dz\right\vert  &  =\left\vert \int_{B_{\varepsilon}(0)}f\left(
z\right)  e^{-\frac{\Vert z\Vert_{d}^{2}}{2t}}dz+\int_{\mathbb{R}%
^{d}\backslash B_{\varepsilon}(0)}f\left(  z\right)  e^{-\frac{\Vert
z\Vert_{d}^{2}}{2t}}dz\right\vert \label{basicbound}\\
&  \leq C_{0}(2\pi t)^{d/2}t^{k/2}2^{k/2}\frac{\Gamma((k+d)/2)}{\Gamma
(d/2)}+O(e^{-\frac{\varepsilon}{2t}})\nonumber
\end{align}
for some constant $C_{0}>0$, and similarly
\begin{equation}
\left|\int_{\mathbb{R}^{d}}f\left(  z\right)  e^{-\frac{\Vert z\Vert_{d}^{2}}{4t}%
}dz\right|\leq(2\pi t)^{d/2}t^{k/2}2^{\left(  2k-d\right)  /2}\frac{\Gamma
((k+d)/2)}{\Gamma(d/2)}+O\left(e^{-\frac{\varepsilon}{4t}}\right).\label{bb2}%
\end{equation}

We now state the main lemma in this section

\begin{lemma}[Odd-cancellation]
\label{local_moments} Let $\{i_{1},\ldots,i_{p}\}$ be a subset of
$\{1,\ldots,d\}$ and let $\alpha_{i_{1}},\ldots,\alpha_{i_{p}}$ be positive
integers and let $\alpha=\alpha_{i_{1}}+\ldots+\alpha_{i_{p}}$. Then for any
$\varepsilon>0$ we have

\begin{enumerate}
\item If all $\alpha_{i_{j}},\ldots,\alpha_{i_{p}}$ are even then
\begin{align*}
\frac{1}{(2\pi t)^{d/2}}\int_{B^{d}_{\varepsilon}(0)}e^{-\frac{\Vert z\Vert
_{d}^{2}}{2t}}z_{i_{1}}^{\alpha_{i_{1}}}\cdot\ldots\cdot z_{i_{p}}%
^{\alpha_{i_{p}}}dz=t^{\alpha/2}\cdot\omega_{\alpha_{i_{1}},\ldots
,\alpha_{i_{p}}}+O\left(  t^{\alpha/2}\cdot e^{-\frac{\varepsilon^{2}}{4t}%
}\right)  .
\end{align*}

\item If at least one $\alpha_{i_{i}},\ldots,\alpha_{i_{p}}$ is odd
\begin{align}
\frac{1}{(2\pi t)^{d/2}}\int_{B^{d}_{\varepsilon}(0)}e^{-\frac{\Vert z\Vert
_{d}^{2}}{2t}}z_{i_{1}}^{\alpha_{i_{1}}}\cdot\ldots\cdot z_{i_{p}}%
^{\alpha_{i_{p}}}dz=0.
\end{align}

\end{enumerate}
\end{lemma}

\subsection{Proofs of Lemmas}

\begin{proof}
[Proof of Lemma \ref{basic_monomials}]The computation of the integral on the
left-hand side of \eqref{intRd} is in fact equivalent to computing a moment of
order $\alpha$ of a multivariate normal random vector. More specifically, let
$\xi_{1},\ldots,\xi_{d}$ be independent and identically distributed (i.i.d.)
random variables and assume that the common distribution of these random
variables is $\mathcal{N}(0,1)$, i.e., a normal random variable with mean 0
and variance 1. In this case, the random vector $\xi=(\xi_{1},\ldots,\xi_{d})$
has a multivariate normal distribution with mean vector $\mu=(0,\ldots,0)$ and
co-variance matrix $I_{d}$. Using that the components of the random vector
$\xi$ are independent we have
\begin{align*}
\frac{1}{(2\pi)^{d/2}}\int_{\mathbb{R}^{d}}e^{-\frac{\Vert z\Vert_{d}^{2}}{2}%
}z_{i_{1}}^{\alpha_{i_{1}}}\cdot\ldots\cdot z_{i_{p}}^{\alpha_{i_{p}}}dz &
=\mathbb{E}\left(  \xi_{i_{1}}^{\alpha_{i_{1}}}\cdot\ldots\cdot\xi_{i_{p}%
}^{\alpha_{i_{p}}}\right)  \\
&  =\mathbb{E}\left(  \xi_{i_{1}}^{\alpha_{i_{1}}}\right)  \cdots\ldots
\cdots\mathbb{E}\left(  \xi_{i_{p}}^{\alpha_{i_{p}}}\right)
\end{align*}
where $\mathbb{E}$ is the expectation with respect to the multivariate normal
distribution $\mathcal{N}(\mu,I_{d})$. In particular we may view each factor
$\mathbb{E}(\xi_{j}^{\alpha_{i_{j}}})$ as an expectation with respect to the
one-dimensional normal random variable $\mathcal{N}(0,1)$. Now, for
$j=1,\ldots,d,$ we have
\[
\mathbb{E}\left(  \xi_{i_{1}}^{\alpha_{i_{1}}}\right)  =\frac{1}{\sqrt{2\pi}%
}\int_{-\infty}^{\infty}e^{-\frac{x^{2}}{2}}x^{\alpha_{i_{j}}}dx.
\]
Observe that by symmetry, if $\alpha_{i_{j}}$ is odd we have $\mathbb{E}%
\left(  \xi_{i_{1}}^{\alpha_{i_{1}}}\right)  =0$. On the other hand, if
$\alpha_{i_{j}}$ is even we write $\alpha_{i_{j}}=2k_{i_{j}}$ so that the
computation of $\mathbb{E}\left(  \xi_{i_{1}}^{\alpha_{i_{1}}}\right)  $
reduces to the integral of $x^{2k_{i_{j}}}e^{-\frac{x^{2}}{2}}$, or in other
words, the $2k_{i_{j}}$ moment of an univariate normal random variable with
mean zero and variance 1. Even though this computation is well-known, we
illustrate a simple way to compute this moment for the reader's convenience.
Observe that for the standard normal random variable we have
\[
\frac{1}{(2\pi)^{1/2}}\int_{-\infty}^{\infty}e^{-\frac{x^{2}}{2}}dx=1,
\]
and by scaling, if we consider a parameter $s>0$ we have
\[
\int_{-\infty}^{\infty}e^{-\frac{sx^{2}}{2}}dx=\sqrt{2\pi}s^{-1/2}%
\]
and differentiating both sides $k_{i_{j}}$ times and evaluating at $s=1$ we
obtain
\begin{align*}
\frac{(-1)^{k}}{2^{k}}\int_{-\infty}^{\infty}x^{2k_{i_{j}}}e^{-x^{2}/2}dx &
=(-1)^{k}\sqrt{2\pi}\frac{1}{2}\cdot\frac{3}{2}\cdot\ldots\cdot\frac
{2k_{i_{j}}-1}{2}\\
&  =\frac{(-1)^{k}}{2^{k}}\sqrt{2\pi}(2k_{i_{j}}-1)!!
\end{align*}
We obtain $\mathbb{E}(\xi_{i_{j}}^{\alpha_{i_{j}}})=(2k_{i_{j}}-1)!!$ which
clearly implies the rest of the lemma.
\end{proof}

\begin{proof}
[Proof of Lemma \ref{local_moments}]We start by writing
\begin{align*}
&  \frac{1}{(2\pi t)^{d/2}}\int_{B^{d}_{\varepsilon}(0)}e^{-\frac{\|z\|_{d}^{2}%
}{2t}}z_{i_{1}}^{\alpha_{i_{1}}}\cdot\ldots\cdot z^{\alpha_{i_{p}}}_{i_{p}%
}dz\\
&  =\frac{1}{(2\pi t)^{d/2}}\int_{\mathbb{R}^{d}}e^{-\frac{\|z\|^{2}_{d}}{2t}%
}z_{i_{1}}^{\alpha_{i_{1}}}\cdot\ldots\cdot z^{\alpha_{i_{p}}}_{i_{p}}%
dz+\frac{1}{(2\pi t)^{d/2}}\int_{\mathbb{R}^{d}\setminus B^{d}_{\varepsilon}%
(0)}e^{-\frac{\|z\|^{2}_{d}}{2t}}z_{i_{1}}^{\alpha_{i_{1}}}\cdot\ldots\cdot
z^{\alpha_{i_{p}}}_{i_{p}}dz\\
&  =\mathrm{I}+\mathrm{II}%
\end{align*}

By scaling, i.e., using $z=\sqrt{t}w$
\begin{align}
\mathrm{I}  &  =\frac{1}{(2\pi t)^{d/2}}\int_{\mathbb{R}^{d}}e^{-\frac{\Vert
z\Vert_{d}^{2}}{2t}}z_{i_{1}}^{\alpha_{i_{1}}}\cdot\ldots\cdot z_{i_{p}%
}^{\alpha_{i_{p}}}dz\nonumber\\
&  =\frac{1}{(2\pi t)^{d/2}}\int_{\mathbb{R}^{d}}e^{-\frac{\Vert w\Vert
_{d}^{2}}{2}}\left(  \sqrt{t}w\right)  _{i_{1}}^{\alpha_{i_{1}}}\cdot
\ldots\cdot\left(  \sqrt{t}w\right)  _{i_{p}}^{\alpha_{i_{p}}}t^{d/2}%
dw.\nonumber\\
&  =\frac{t^{\frac{\alpha}{2}}}{(2\pi)^{d/2}}\int_{\mathbb{R}^{d}}%
e^{-\frac{\Vert w\Vert_{d}^{2}}{2}}w_{i_{1}}^{\alpha_{i_{1}}}\cdot\ldots\cdot
w_{i_{p}}^{\alpha_{i_{p}}}dw=t^{\alpha/2}\omega_{\alpha_{i_{1}},\ldots
,\alpha_{i_{p}}} \label{intgauss}%
\end{align}
Note that \eqref{intgauss} follows from Lemma \eqref{basic_monomials}. We now
need to estimate $\mathrm{II}$. Observe that from equation \eqref{rescaled} we
have
\begin{align*}
\left\vert \mathrm{II}\right\vert  &  =\frac{1}{(2\pi t)^{d/2}}\left\vert
\int_{\mathbb{R}^{d}\setminus B^{d}_{\varepsilon}(0)}e^{-\frac{\Vert z\Vert
_{d}^{2}}{2t}}z_{i_{1}}^{\alpha_{i_{1}}}\cdot\ldots\cdot z_{i_{p}}%
^{\alpha_{i_{p}}}dz\right\vert \\
&  =\frac{1}{(2\pi t)^{d/2}}\left\vert \int_{\mathbb{R}^{d}\setminus
B^{d}_{\varepsilon}(0)}e^{-\frac{\Vert z\Vert_{d}^{2}}{4t}-\frac{\Vert z\Vert
_{d}^{2}}{4t}}z_{i_{1}}^{\alpha_{i_{1}}}\cdot\ldots\cdot z_{i_{p}}%
^{\alpha_{i_{p}}}dz\right\vert \\
&  \leq\frac{e^{-\frac{\varepsilon}{4t}}}{(2\pi t)^{d/2}}\int_{\mathbb{R}^{d}%
}e^{-\frac{\Vert z\Vert_{d}^{2}}{4t}}\Vert z\Vert^{\alpha}dz\\
&  =e^{-\frac{\varepsilon}{4t}}2^{(2\alpha+d)/2}t^{\alpha/2}\left(\frac{\Gamma((k+\alpha)/2)}{\Gamma(\alpha/2)}\right)+O\left(e^{-\frac{\varepsilon}{4t}}\right)
\end{align*}
This proves the lemma.
\end{proof}

\section{\bigskip Proof of Expansion Lemma \ref{BigNasty}}\label{appendixB}

\begin{proof}[Proof of Lemma \ref{BigNasty}]
Recall the following expansions (cf. equation \eqref{l98})
\[
e^{\frac{\phi(z)}{2t}}=1+\frac{p_{4}(\phi,x)[z]+p_{5}(\phi,x)[z]}{2t}%
+\frac{\varrho_{6}(\phi,x)[z]}{2t}+\varrho(\phi,x,t)[z]
\]
and
\[
d\mu=\left(  1+\sum_{j=2}^{m}p_{j}(d\mu,x)[z]+\varrho_{m+1}(d\mu,x)[z]\right)
dz.
\]
Then%
\begin{subequations}
\begin{equation}
\frac{1}{(2\pi t)^{d/2}}\int_{B_{\varepsilon}^{d}(0)}e^{-\frac{\Vert
x-\exp_{x}z\Vert_{N}^{2}}{2t}}\tilde{w}(x,z)d\mu(z)=\frac{1}{(2\pi t)^{d/2}%
}\int_{B_{\varepsilon}^{d}(0)}e^{-\frac{\Vert z\Vert_{d}^{2}}{2t}}%
e^{\frac{\phi(z)}{2t}}\tilde{w}(x,z)d\mu(z)\label{Big1}%
\end{equation}
by definition of $\phi(z).$ We now look at expansion of the integrand, Taylor expanding $\tilde{w}(x,z)$
in terms of $z$ at $z=0.$ 
\end{subequations}
\[
e^{\frac{\phi(z)}{2t}}\tilde{w}(x,z)d\mu=\left\{
\begin{array}
[c]{c}%
\left(  1+\frac{p_{4}(\phi,x)[z]+p_{5}(\phi,x)[z]}{2t}+\frac{\varrho_{6}%
(\phi,x)[z]}{2t}+\varrho(\phi,x,t)[z]\right)  \times\\
\left(  \displaystyle{\sum_{j=0}^{4}}D_{z}^{j}\tilde{w}(x,0)+\varrho_{5}(\tilde
{w},x)[z]\right)  \times\\
\left(  1+\displaystyle{\sum_{j=2}^{4}}p_{j}(d\mu,x)[z]+\varrho_{5}(d\mu,x)[z]\right)
\end{array}
\right\}dz.
\]
We group these by order of polynomial in $z$, and order of remainder. We
distinguish between terms that have the factor of $1/(2t)$ and those that do
not. First,
\begin{align}
P_{0}(x,z) &  =\tilde{w}(x,0)
\end{align}
\begin{align}
P_{1}(x,z) &  =D_{z}^{1}\tilde{w}(x,0)[z]
\end{align}
\begin{align}
P_{2}(x,z) &  =D_{z}^{2}\tilde{w}(x,0)[z]+\tilde{w}(x,0)p_{2}(d\mu,x)[z]
\end{align}
\begin{align}
P_{3}(x,z) &  =D_{z}^{3}\tilde{w}(x,0)[z]+D_{z}^{1}\tilde{w}(x,0)\otimes
p_{2}(d\mu,x)[z]+\tilde{w}(x,0)p_{3}(d\mu,x)[z]
\end{align}
\begin{align}
P_{4}(x,z) &  =D_{z}^{4}\tilde{w}(x,0)[z]+D_{z}^{2}\tilde{w}(x,0)\otimes
p_{2}(d\mu,x)[z]\nonumber\\
&  +D_{z}^{1}\tilde{w}(x,0)\otimes p_{3}(d\mu,x)[z]+\tilde{w}(x,0)p_{4}%
(d\mu,x)[z]
\end{align}
\begin{align}
P_{5}(x,z) &  =D_{z}^{3}\tilde{w}(x,0)\otimes p_{2}(d\mu,x)[z]+D_{z}^{2}%
\tilde{w}(x,0)\otimes p_{3}(d\mu,x)[z]\nonumber\\
&  +D_{z}^{1}\tilde{w}(x,0)\otimes p_{4}(d\mu,x)[z]
\end{align}
\begin{align}
P_{6}(x,z) &  =D_{z}^{4}\tilde{w}(x,0)\otimes p_{2}(d\mu,x)[z]+D_{z}^{3}%
\tilde{w}(x,0)\otimes p_{3}(d\mu,x)[z]\nonumber\\
&  +D_{z}^{2}\tilde{w}(x,0)\otimes p_{4}(d\mu,x)[z]
\end{align}
\begin{align}
P_{7}(x,z) &  =D_{z}^{4}\tilde{w}(x,0)\otimes p_{3}(d\mu,x)[z]+D_{z}^{3}%
\tilde{w}(x,0)\otimes p_{4}(d\mu,x)[z]
\end{align}
\begin{align}
P_{8}(x,z) &  =D_{z}^{4}\tilde{w}(x,0)\otimes p_{4}(d\mu,x)[z].
\end{align}
Next, we list all homogeneous terms containing the factor $\frac{1}{2t}$
\begin{align}
\tilde{P}_{4}(x,z) &  =\tilde{w}(x,0)\frac{p_{4}(\phi,x)[z]}{2t}
\end{align}
\begin{align}
\tilde{P}_{5}(x,z) &  =\tilde{w}(x,0)\frac{p_{5}(\phi,x)[z]}{2t}+D_{z}%
^{1}\tilde{w}(x,0)\otimes p_{4}(\phi,x)[z]
\end{align}
\begin{align}
\tilde{P}_{6} &  =\frac{1}{2t}\left\{
\begin{array}
[c]{c}%
\tilde{w}(x,0)p_{4}(\phi,x)\otimes p_{2}(d\mu,x)[z]+p_{4}(\phi,x)\otimes
D_{z}^{2}\tilde{w}(x,0)[z]\\
+\frac{1}{2t}p_{5}(\phi,x)\otimes D_{z}^{1}\tilde{w}(x,0)[z]
\end{array}
\right\}  \\
\tilde{P}_{7} &  =\frac{1}{2t}\left\{
\begin{array}
[c]{c}%
\tilde{w}(x,0)p_{4}(\phi,x)\otimes p_{3}(d\mu,x)[z]+p_{4}(\phi,x)\otimes
D_{z}^{3}\tilde{w}(x,0)[z]\\
+p_{4}(\phi,x)\otimes p_{2}(d\mu,x)\otimes D_{z}\tilde{w}(x,0)[z]\\
+p_{5}(\phi,x)\otimes D_{z}^{2}\tilde{w}(x,0)+p_{5}(\phi,x)\otimes p_{2}%
(d\mu,x)[z]
\end{array}
\right\}
\end{align}%
\begin{align}
\tilde{P}_{8} &  =\frac{1}{2t}\left(  p_{4}(\phi,x)+p_{5}(\phi,x)\right)
\otimes\left(  \sum_{j=0}^{4}D_{z}^{j}\tilde{w}(x,0)\right)  \otimes\left(
1+\sum_{j=2}^{4}p_{j}(d\mu,x)\right)  [z]\\
&  -\sum_{i=0}^{7}\tilde{P}_{i}(z)
\end{align}
and remainder terms
\begin{align}
R_{\phi}(x,z,t) &  =\left[  \frac{\varrho_{6}(\phi,x)[z]}{2t}+\varrho
(\phi,x,t)[z]\right]  \left(  \sum_{j=0}^{4}D_{z}^{j}\tilde{w}(x,0)\right)
\otimes\left(  1+\sum_{j=2}^{4}p_{j}(d\mu,x)\right)  [z]\label{bterms3}\\
R_{w}(x,z,t) &  =\varrho_{5}(\tilde{w},x)[z]\left(  1+\frac{p_{4}%
(\phi,x)+p_{5}(\phi,x)}{2t}\right)  \otimes\left(  1+\sum_{j=2}^{4}p_{j}%
(d\mu,x)\right)  [z]\label{AOC}\\
R_{\mu}(x,z,t) &  =\varrho_{5}(d\mu,x)[z]\left(  1+\frac{p_{4}(\phi
,x)+p_{5}(\phi,x)}{2t}\right)  \otimes\left(  \sum_{j=0}^{4}D_{z}^{j}\tilde
{w}(x,0)\right)  [z]\label{TheMooch}\\
R_{\phi w}(x,z,t) &  =\left[  \frac{\varrho_{6}(\phi,x)[z]}{2t}+\varrho
(\phi,x,t)[z]\right]  \varrho_{5}(\tilde{w},x)[z]\left(  1+\sum_{j=2}^{4}%
p_{j}(d\mu,x)\right)  [z]\\
R_{\phi\mu}(x,z,t) &  =\left[  \frac{\varrho_{6}(\phi,x)[z]}{2t}+\varrho
(\phi,x,t)[z]\right]  \varrho_{5}(d\mu,x)[z]\left(  \sum_{j=0}^{4}D_{z}%
^{j}\tilde{w}(x,0)\right)  [z]\\
R_{w\mu}(x,z,t) &  =\varrho_{5}(\tilde{w},x)[z]\varrho_{5}(d\mu,x)[z]\left(
1+\frac{p_{4}(\phi,x)[z]+p_{5}(\phi,x)[z]}{2t}\right)  \\
R_{\phi w\mu}(x,z,t) &  =\left[  \frac{\varrho_{6}(\phi,x)[z]}{2t}%
+\varrho(\phi,x,t)[z]\right]  \varrho_{5}(\tilde{w},x)[z]\varrho_{5}(d\mu
,x)[z].\label{Tulsi}%
\end{align}
With this setup (\ref{Big1}) becomes%
\begin{equation}
\frac{1}{(2\pi t)^{d/2}}\int_{B_{\varepsilon}^{d}(0)}e^{-\frac{\Vert
z\Vert_{d}^{2}}{2t}}\left\{
\begin{array}
[c]{c}%
\displaystyle{\sum_{i=0}^{8}}P_{i}(x,z)+\displaystyle{\sum_{i=4}^{7}}\tilde{P}_{i}(x,z)+\tilde{P}_{8}(x,z)\\
+\displaystyle{\sum_{\clubsuit\in\left\{  \phi,w.\mu,\phi w,w\mu,\phi\mu w\right\}
}}R_{\clubsuit}(x,z,t)
\end{array}
\right\}  dz.\label{checkpoint}%
\end{equation}
We deal with these in turn. First, applying Lemma \ref{local_moments} we have
\begin{align}
&\frac{1}{(2\pi t)^{d/2}}\int_{B_{\varepsilon}^{d}(0)}e^{-\frac{\Vert
z\Vert_{d}^{2}}{2t}}\left(  \sum_{i=0}^{8}P_{i}(x,z)\right)  dz\nonumber\\
&=\frac{1}{(2\pi)^{d/2}}\int_{%
\mathbb{R}
^{d}}e^{-\frac{\Vert\zeta\Vert_{d}^{2}}{2}}\left\{
\begin{array}
[c]{c}%
\tilde{w}(x,0)\\
+\left\{  D_{z}^{2}\tilde{w}(x,0)[\zeta]+\tilde{w}(x,0)p_{2}(d\mu
,x)[\zeta]\right\}  t\\
+\left\{
\begin{array}
[c]{c}%
D_{z}^{4}\tilde{w}(x,0)[\zeta]+D_{z}^{2}\tilde{w}(x,0)\otimes p_{2}%
(d\mu,x)[\zeta]\\
+D_{z}^{1}\tilde{w}(x,0)\otimes p_{3}(d\mu,x)[\zeta]+\tilde{w}(x,0)\otimes
p_{4}(d\mu,x)[\zeta]
\end{array}
\right\}  t^{2}\\
+\left\{
\begin{array}
[c]{c}%
D_{z}^{4}\tilde{w}(x,0)\otimes p_{2}(d\mu,x)[\zeta]+\\
+D_{z}^{3}\tilde{w}(x,0)\otimes p_{3}(d\mu,x)[\zeta]+D_{z}^{4}\tilde
{w}(x,0)\otimes p_{2}(d\mu,x)[\zeta]
\end{array}
\right\}  t^{3}\\
+\left\{D_{z}^{4}\tilde{w}(x,0)[\zeta]\otimes p_4(d\mu,x)[\zeta]\right\}t^{4}%
\end{array}
\right\}  d\zeta\label{bf1}\\
&  +\sum_{\alpha=0}^{8}O\left(  t^{\alpha/2}e^{-\frac{\varepsilon^{2}}{4t}%
};\Sigma,J^{4}w,J^{4}d\mu\right)  .
\end{align}
Next, repeating for
\begin{align}
&  \frac{1}{(2\pi t)^{d/2}}\int_{B_{\varepsilon}^{d}(0)}e^{-\frac{\Vert
z\Vert_{d}^{2}}{2t}}\sum_{i=4}^{7}\tilde{P}_{i}(x,z)dz\nonumber\\
&  =\frac{1}{(2\pi)^{d/2}}\int_{%
\mathbb{R}
^{d}}e^{-\frac{\Vert\zeta\Vert_{d}^{2}}{2}}\left\{
\begin{array}
[c]{c}%
\tilde{w}(x,0)\frac{p_{4}(\phi,x)[\zeta]}{2t}t^{2}\\
+\left(
\begin{array}
[c]{c}%
\frac{p_{4}(\phi,x)}{2t}\otimes\left[  \tilde{w}(x,0)p_{2}(d\mu,x)[\zeta
]+D_{z}^{2}\tilde{w}(x,0)[\zeta]\right]  \\
+\frac{p_{5}(\phi,x)}{2t}\otimes D_{z}\tilde{w}(x,0)[\zeta]
\end{array}
\right)  t^{3}%
\end{array}
\right\}  d\zeta\label{bf2}\\
&  +\frac{1}{2t}\sum_{\alpha=4}^{6}O\left(  t^{\alpha/2}e^{-\frac
{\varepsilon^{2}}{4t}};\Sigma,J^{2}w,J^{2}d\mu,J^{5}\phi\right)  \\
&  =\frac{1}{(2\pi)^{d/2}}\int_{%
\mathbb{R}
^{d}(0)}\frac{1}{2}e^{-\frac{\Vert\zeta\Vert_{d}^{2}}{2}}\left\{
\begin{array}
[c]{c}%
p_{4}(\phi,x)\tilde{w}(x,0)[\zeta]t\\
+\left(
\begin{array}
[c]{c}%
p_{4}(\phi,x)\otimes\left[  \tilde{w}(x,0)p_{2}(d\mu,x)[\zeta]+D_{z}^{2}%
\tilde{w}(x,0)[\zeta]\right]  \\
+p_{5}(\phi,x)\otimes D_{z}\tilde{w}(x,0)[\zeta]
\end{array}
\right)  t^{2}%
\end{array}
\right\}d\zeta  \\
&  +\sum_{\alpha=2}^{4}O\left(  t^{\alpha/2}e^{-\frac{\varepsilon^{2}}{4t}%
};\Sigma,J^{2}w,J^{2}d\mu,J^{5}\phi\right)  .
\end{align}

At this point we pause, and collecting terms, observe the fact that%
\begin{align*}
&  \frac{1}{(2\pi t)^{d/2}}\int_{B_{\varepsilon}^{d}(0)}e^{-\frac{\Vert
z\Vert_{d}^{2}}{2t}}\left\{  \sum_{i=0}^{8}P_{i}(x,z)+\sum_{i=4}^{7}\tilde
{P}_{i}(x,z)\right\}  dz.\\
&  =\frac{1}{(2\pi)^{d/2}}\int_{%
\mathbb{R}
^{d}}e^{-\frac{\Vert\zeta\Vert_{d}^{2}}{2}}\left\{
\begin{array}
[c]{c}%
\tilde{w}(x,0)+W_{1}(x,\zeta)t\\
+W_{2}(x,\zeta)t^{2}+W_{3}(x,\zeta)t^{3}+W_{4}(x,\zeta)t^{4}%
\end{array}
\right\}  d\zeta\\
&  +O\left(  te^{-\frac{\varepsilon^{2}}{4t}};\Sigma,J^{2}w,J^{2}d\mu,J^{5}
\phi\right)  +O\left(  e^{-\frac{\varepsilon^{2}}{4t}};\Sigma,J^{4}w,J^{4}
d\mu\right)  .\\
&=\frac{1}{(2\pi)^{d/2}}\int_{%
\mathbb{R}
^{d}}e^{-\frac{\Vert\zeta\Vert_{d}^{2}}{2}}\left\{
\begin{array}
[c]{c}%
\tilde{w}(x,0)+W_{1}(x,\zeta)t\\
+W_{2}(x,\zeta)t^{2}+W_{3}(x,\zeta)t^{3}+W_{4}(x,\zeta)t^{4}%
\end{array}
\right\}  d\zeta\\
&+O\left(  e^{-\frac{\varepsilon^{2}}{4t}};\Sigma,J^{4}w,J^{4}%
d\mu,J^{5}\phi\right).
\end{align*}
In other words, going back to (\ref{checkpoint}) what is needed to complete
the proof is to show that%
\begin{align}
& \frac{1}{(2\pi t)^{d/2}}\int_{B_{\varepsilon}^{d}(0)}e^{-\frac{\Vert
z\Vert_{d}^{2}}{2t}}\left\{  \tilde{P}_{8}(x,z)+\sum_{\clubsuit\in\left\{
\phi,w.\mu,\phi w,w\mu,\phi\mu w\right\}  }R_{\clubsuit}(x,z,t)\right\}  dz\\
& =\tilde{w}(x,0)\Phi(t^{2};\phi)+O(t^{5/2};\Sigma,J^{5}w,J^{4}d\mu,J^{5}\phi).
\end{align}
First we argue that%
\begin{equation}
\frac{1}{(2\pi t)^{d/2}}\int_{B_{\varepsilon}^{d}(0)}e^{-\frac{\Vert
z\Vert_{d}^{2}}{2t}}\tilde{P}_{8}(x,z)dz=O(t^{3};\Sigma,\allowbreak J^{4}w,J^{4}d\mu,J^{5}%
\phi).\label{Qanon}%
\end{equation}
Notice the expansion of $\tilde{P}_{8}$ involves 40 terms, before subtracting
the lower order terms, so we elect to not express these explicitly. However,
for fixed $x,$ $\tilde{P}_{8}$ will be a polynomial in $z$ with a factor of
$1/t$, which can be made explicit in terms $\allowbreak$of $J^{4}w,J^{5}\phi$
and $J^{4}d\mu.$ Each term has order at least $8$ in $z.$ Using the change
of variable $z=\sqrt{t}\zeta$ as in the proof of Lemma (\ref{local_moments})
gives that each term has order at least 4 in $t.$ Since there was an initial
factor of $1/t$ we conclude each term is bounded by a term of order $O(t^{3};J^{4}w,J^{4}d\mu,J^{5}\phi)$. From Lemma (\ref{local_moments}) the remainder
not yet accounted for is of order $O(t^{3}e^{-\varepsilon^{2}/4t};J^{4}w,J^{4}d\mu,J^{5}\phi)$. \ We combine to conclude (\ref{Qanon}.)
\\

Now we turn to remainder terms
\begin{align}
& \frac{1}{(2\pi t)^{d/2}}\int_{B_{\varepsilon}^{d}(0)}e^{-\frac{\Vert
z\Vert_{d}^{2}}{2t}}\left\{  \sum_{\clubsuit\in\left\{  \phi,w.\mu,\phi
w,w\mu,\phi\mu w\right\}  }R_{\clubsuit}(x,z,t)\right\}  dz\label{BenShapiro}%
\\
& =\tilde{w}(x,0)\Phi(t^{2};\phi)(x)+O(t^{5/2};J^{5}w,J^{5}d\mu,J^{5}\phi)
\end{align}
\ We will explicitly deal with $R_{\phi}(x,z,t),$ the remaining $7$ terms can
be dealt with in a similar fashion: \
\[
\frac{1}{(2\pi t)^{d/2}}\int_{B_{\varepsilon}^{d}(0)}e^{-\frac{\Vert
z\Vert_{d}^{2}}{2t}}R_{\phi}(x,z,t)dz
\]%
\begin{align*}
&  =\frac{1}{(2\pi t)^{d/2}}\int_{B_{\varepsilon}^{d}(0)}e^{-\frac{\Vert
z\Vert_{d}^{2}}{2t}}\left[  \frac{\varrho_{6}(\phi,x)[z]}{2t}+\varrho
(\phi,x,t)[z]\right]  \\
&  \times\left(  \sum_{j=0}^{4}D_{z}^{j}\tilde{w}(x,0)\right)  \otimes\left(
1+\sum_{j=2}^{4}p_{j}(d\mu,x)\right)  [z]dz
\end{align*}
Recall
\begin{equation}
\varrho_{6}(\phi,x)[z]=O(\left\Vert z\right\Vert ^{6};\Sigma,J^{6}\phi)\label{bb5}%
\end{equation}
and (\ref{order8})%
\begin{equation}
e^{-\frac{\Vert z\Vert_{d}^{2}}{2t}}\left\vert \varrho(\phi,x,t)[z]\right\vert
\leq C_{\Sigma}^{2}\frac{\Vert z\Vert_{d}^{8}}{8t^{2}}e^{-\frac{\Vert
z\Vert_{d}^{2}}{4t}}.\label{bb6}%
\end{equation}
We then get
\begin{align*}
&  \frac{1}{(2\pi t)^{d/2}}\int_{B_{\varepsilon}^{d}(0)}e^{-\frac{\Vert
z\Vert_{d}^{2}}{2t}}\left[  \frac{\varrho_{6}(\phi,x)[z]}{2t}+\varrho
(\phi,x,t)[z]\right]  \\
&  \times\left(  \sum_{j=0}^{4}D_{z}^{j}\tilde{w}(x,0)\right)  \otimes\left(
1+\sum_{j=2}^{4}p_{j}(d\mu,x)\right)  [z]dz\\
&  \leq\frac{1}{(2\pi t)^{d/2}}\int_{B_{\varepsilon}^{d}(0)}e^{-\frac{\Vert
z\Vert_{d}^{2}}{2t}}\left[  C_{\phi}\left\Vert z\right\Vert ^{6}\right]
\left(  \sum_{j=0}^{4}D_{z}^{j}\tilde{w}(x,0)\right)  \otimes\left(
1+\sum_{j=2}^{4}p_{j}(d\mu,x)\right)  \allowbreak\lbrack z]dz\\
&  +\frac{1}{(2\pi t)^{d/2}}\int_{B_{\varepsilon}^{d}(0)}C_{\Sigma}^{2}%
\frac{\Vert z\Vert_{d}^{8}}{8t^{2}}e^{-\frac{\Vert z\Vert_{d}^{2}}{4t}}\left(
\sum_{j=0}^{4}D_{z}^{j}\tilde{w}(x,0)\right)  \otimes\left(  1+\sum_{j=2}%
^{4}p_{j}(d\mu,x)\right)  [z]dz
\end{align*}
where $C_{\phi}$ is a constant that depends on bounds on $J^{6}\phi$. Scaling, we obtain the following upper bound
\begin{align}
& \frac{1}{(2\pi)^{d/2}}\int_{B_{\frac{\varepsilon}{\sqrt{t}}}^{d}(0)}%
e^{-\frac{\Vert\zeta\Vert_{d}^{2}}{2}}\left[  C_{\phi}\left\Vert \zeta
\right\Vert ^{6}t^{3}\right]  \nonumber\\
&  \times\left(  \sum_{j=0}^{4}D_{z}^{j}\tilde{w}(x,0)t^{j/2}\right)
\otimes\left(  1+\sum_{j=2}^{4}p_{j}(d\mu,x)t^{j/2}\right)  [\zeta]d\zeta\\
&  +\frac{1}{(2\pi)^{d/2}}\int_{B_{\varepsilon t^{-1/2}}^{d}(0)}C_{\Sigma}%
^{2}\frac{\Vert\zeta\Vert_{d}^{8}t^{2}}{8}e^{-\frac{\Vert\zeta\Vert_{d}^{2}%
}{4}}\nonumber\\
&  \times\left(  \sum_{j=0}^{4}D_{z}^{j}\tilde{w}(x,0)t^{j/2}\right)
\otimes\left(  1+\sum_{j=2}^{4}p_{j}(d\mu,x)t^{j/2}\right)  [\zeta]d\zeta.
\end{align}
Now clearly, the first term (in the sense by Lemma \eqref{local_moments}) is a sum of
terms with order at least $3$ in $t$ and dependence on $J^{4}w$ and $J^{4}%
d\mu.$ The second term has a single term of order 2 in $t,$ namely%
\begin{equation}
\frac{1}{(2\pi)^{d/2}}\int_{B^{d}_{\frac{\varepsilon}{\sqrt{t}}}(0)}C_{\Sigma}%
^{2}\frac{\Vert\zeta\Vert_{d}^{8}t^{2}}{8}e^{-\frac{\Vert\zeta\Vert_{d}^{2}%
}{4}}\tilde{w}(x,0)d\zeta=\tilde{w}(x,0)\cdot\Phi(t^{2};\phi).\nonumber
\end{equation}
The remainder of the terms (again using the estimates in Lemma (\ref{local_moments})) are
terms of higher order. We have shown (\ref{BenShapiro}). The analysis of each of the remainder terms of the form 
\begin{equation}
\frac{1}{(2\pi t)^{d/2}}\int_{B_{\varepsilon}^{d}(0)}e^{-\frac{\Vert
z\Vert_{d}^{2}}{2t}}R_{\clubsuit}(x,z,t)dz\label{bterms2}%
\end{equation}
is similar: In each case we can use bounds such as (\ref{bb5}) or (\ref{bb6})
and in others we can use bounds such as
\begin{align*}
\varrho_{5}(\tilde{w},x)[z] = O(\|z\|^{5}_{d};\Sigma,J^{5}\tilde{w}(x,z)).
\end{align*}
We note that after integration, lines \eqref{bterms3}, \eqref{AOC} and \eqref{TheMooch} 
determine the least decaying terms. In fact, each of these terms decays like $O(t^{5/2})$, 
but the term \eqref{bterms3} contains a term of order 6 in $z$ with a factor of $1/(2t)$ which of course comes down to a term of order $O(t^2)$ after rescaling. This
is why this term needed special consideration and why we see the term
$\tilde{w}(x,0){\Phi}(t^{2};\phi)$ in the expansion. In fact, the term \eqref{bterms3} is equivalent to $\tilde{w}(x,0){\Phi}(t^{2};\phi)$ in the sense of inequality \eqref{defPhi}.
\end{proof}

\label{appendix}

\bibliographystyle{amsalpha}
\bibliography{coarse_ricci}

\end{document}